\documentclass[a4paper,12pt,reqno]{amsart}

\usepackage{amsmath,amsfonts,amsthm,amssymb,color}
\usepackage[latin1]{inputenc}
\usepackage{xy} %options graphiques
\usepackage{graphicx}
\usepackage{pdfsync}
\usepackage{listings}
\usepackage{epsfig}
\usepackage{stmaryrd}
\usepackage{hyperref}

\newcommand{\xrgh}{\mathbf{x}}
\newcommand{\der}{\delta}

\newcommand{\cacha}{\Hat{\mathcal{C}}}

\newcommand{\delha}{\hat{\delta}}

\newcommand{\Laha}{\hat{\Lambda}}

\newcommand{\norm}[1]{\lVert #1\rVert}

\newcommand{\yti}{\tilde{y}}
\newcommand{\xti}{\tilde{x}}
\newcommand{\yba}{\overline{y}}

\newcommand{\ka}{\kappa}

\newcommand{\bt}{\mathbf{2}}

\DeclareMathOperator{\id}{\text{Id}}

\DeclareMathOperator{\Id}{Id}

\makeatletter
\newcommand{\eqcolon}{\mathrel{\mathord{=}\raise.2\p@\hbox{:}}}
\newcommand{\coloneq}{\mathrel{\raise.2\p@\hbox{:}\mathord{=}}}
\makeatother

\newcommand{\R}{\mathbb R}
\newcommand{\N}{\mathbb N}

\newcommand{\cb}{\mathcal B}
\newcommand{\cac}{\mathcal C}

\newcommand{\cj}{\mathcal J}
\newcommand{\cl}{\mathcal L}
\newcommand{\cn}{\mathcal N}
\newcommand{\cq}{\mathcal Q}

\newcommand{\cs}{\mathcal S}

\newcommand{\al}{\alpha}
\newcommand{\ep}{\varepsilon}

\newcommand{\ga}{\gamma}
\newcommand{\la}{\lambda}

\newcommand{\vp}{\varphi}

\newcommand{\be}{\beta}

\newcommand{\lp}{\left(}
\newcommand{\rp}{\right)}
\newcommand{\lc}{\left[}
\newcommand{\rc}{\right]}
\newcommand{\lcl}{\left\{}
\newcommand{\rcl}{\right\}}
\newcommand{\lln}{\left|}
\newcommand{\rrn}{\right|}
\newcommand{\lla}{\left\langle}
\newcommand{\rra}{\right\rangle}

\newcommand{\bean}{\begin{eqnarray*}}
\newcommand{\eean}{\end{eqnarray*}}
\newcommand{\ben}{\begin{enumerate}}
\newcommand{\een}{\end{enumerate}}
\newcommand{\beq}{\begin{equation}}
\newcommand{\eeq}{\end{equation}}

\newtheorem{theorem}{Theorem}[section]

\newtheorem{corollary}[theorem]{Corollary}

\newtheorem{definition}[theorem]{Definition}
\newtheorem{notation}[theorem]{Notation}

\newtheorem{lemma}[theorem]{Lemma}

\newtheorem{proposition}[theorem]{Proposition}

\theoremstyle{remark}
\newtheorem{remark}[theorem]{Remark}

\begin{document}

\title{Numerical schemes for rough parabolic equations}

\author{Aurélien Deya}

\begin{abstract}
This paper is devoted to the study of numerical approximation schemes for a class of parabolic equations on $(0,1)$ perturbed by a non-linear rough signal. It is the continuation of \cite{RHE,RHE-glo}, where the existence and uniqueness of a solution has been established. The approach combines rough paths methods with standard considerations on discretizing stochastic PDEs. The results apply to a geometric $2$-rough path, which covers the case of the multidimensional fractional Brownian motion with Hurst index $H>1/3$.
\end{abstract}

\address{Aur{\'e}lien Deya, Institut {\'E}lie Cartan Nancy, Universit\'e de Nancy 1, B.P. 239,
54506 Vand{\oe}uvre-l{\`e}s-Nancy Cedex, France}
\email{deya@iecn.u-nancy.fr}

\subjclass[2000]{60H35,60H15,60G22}
\date{\today}
\keywords{Rough paths theory; Stochastic PDEs; Approximation schemes; Fractional Brownian motion}

\maketitle

\section{Introduction}

This paper is part of an ongoing project whose general objective is to adapt the rough paths methods for the study of stochastic partial differential equations. The idea is to extend the concept of a PDE solution so as to handle the case of a non differentiable (and non Wiener-type) driving perturbation. So far, let us say that two kinds of approaches have been considered in this direction. The first one, due to Friz, Caruana, Oberhauser and Diehl (\cite{car-friz-ober,friz-ober-1,friz-ober-2,diehl-friz}), finds its inspiration in the \emph{viscosity}-solution theory for (ordinary) PDEs, and which efficiently combines with the rough paths stability results. The second one, developped by Gubinelli, Tindel and the author (\cite{GT,RHE,RHE-glo}) on the one hand and Teichmann (\cite{teich}) on the other, takes the \emph{mild} formulation of PDEs as the basic model, and then tries to take profit of the semigroup regularizing properties in order to cope with time roughness. In this sense, the latter approach happens to be quite close to the stochastic infinite-dimensional theory by Da Prato and Zabczyk \cite{daprato-zab} (among others), and it shares many characteristics with the recent works of Jentzen, Kloeden and Röckner \cite{jen-2,jen-klo-3,jent-roeck}.

\smallskip

In both \emph{viscosity} and \emph{mild} approaches, the solution of the rough PDE under consideration is obtained by means of theoretical arguments, i.e., either with a fixed-point theorem or an abstract stability result, which give no clue on how to represent this solution. The aim of this paper is to remedy the problem by introducing easily-implementable approximation algorithms. To be more specific, we intend to follow the \emph{mild} formulation of \cite{GT,RHE,RHE-glo} and show that this formalism can be combined with the classical discretization procedures for (Wiener) SPDEs.

\smallskip

To this end, the equation that we will focus on throughout the paper is the following:
\begin{equation}\label{equa-strong}
y_0=\psi \in L^2(0,1) \quad , \quad dy_t=A y_t \, dt+\sum_{i=1}^m f_i(y_t) \, dx^i_t \quad , \quad t\in [0,1],
\end{equation}
where:
\begin{list}{\labelitemi}{\leftmargin=1em\itemsep=1em}
\item $A=\partial_\xi(a \cdot \partial_\xi)+c$ is a Sturm-Liouville operator with Dirichlet boundary conditions on $(0,1)$,
\item $f_i(y_t)(\xi):=f_i(y_t(\xi))$ for some smooth enough function $f_i:\R \to \R$,
\item $x:[0,1] \to \R^m$ is a $\ga$-Hölder path with $\ga >1/3$ which gives rise to a geometric rough path of order $1$ (see Assumption \textbf{(X1)$_\ga$}) or $2$ (see Assumption \textbf{(X2)$_\ga$}).
\end{list}

\

Thanks to the results of \cite{CQ}, we know that the latter hypothesis includes in particular the case where $x$ is a fractional Brownian motion (fBm in the sequel) with Hurst index $H>1/3$. Thus, Equation (\ref{equa-strong}) offers in this situation a model that can deal with the long-range dependance property at the core of many applications in engineering, biophysics or mathematical finance (see for instance \cite{denk-meintrup,kou,rogers}). It is worth mentioning that in the fBm case, the equation can also be handled with Malliavin calculus tools (see \cite{TTV,nu-vui,ss-vui,hu-nu}), but for $H>1/2$ or for very particular choices of $f_i$ only ($f_i=1$ or $f_i=\Id$).

\smallskip

The existence and uniqueness of a mild global solution for (\ref{equa-strong}) has been established in \cite{RHE} when $x$ is a $1$-rough path (\emph{Young case}) and in \cite{RHE-glo} when $x$ is a $2$-rough path (\emph{rough case}). We will of course go back to the exact statement of these two results during the study. The approximation procedure will then stem from two successive discretization steps, in accordance with the strategy displayed for Wiener SPDEs (see \cite{gyon} or \cite{haus}): we first turn to a time-discretization of the problem and then perform a space-discretization of the algorithm, following the Galerkin projection method. Actually, the shape of the schemes will be derived from the very interpretation of the rough term involved in (\ref{equa-strong}). For this reason, let us remind the reader with a few key-points of the approach displayed in \cite{RHE,RHE-glo}:

\

\begin{list}{\labelitemi}{\leftmargin=1em\itemsep=1em}
\item Following the mild formulation of (S)PDEs, the equation is analyzed as
\begin{equation}\label{equa-mild-intro}
y_t=S_t \psi+\sum_{i=1}^m \int_0^t S_{t-u} \, dx^i_u \, f_i(y_u) \quad , \quad t\in [0,1],
\end{equation}
where $S$ stands for the semigroup generated by $A$. This is a classical change of perspective (see \cite{daprato-zab}) and it allows us to resort to the numerous regularizing properties of $S$ (some of these properties are reported in Subsection \ref{section-cadre}).
\item As is the case with rough standard systems, the interpretation of the right-hand-side of (\ref{equa-mild-intro}) relies on the expansion of the convolutional integral $\int_s^t  S_{t-u} \, dx^i_u \, f_i(y_u)$. This expansion gives rise to a decomposition such as
\begin{equation}\label{int-convol-intro}
\int_s^t S_{t-u} \, dx^i_u \, f_i(y_u)=P_{ts}+R_{ts},
\end{equation}
where $P$ is a "main" term and $R$ a "residual" term with high regularity in the time parameters $(s,t)$, $R$ being thus likely to disappear from an infinitesimal point of view (see (\ref{decompo-int-young}) and (\ref{dec-int-rou-reg}) for examples of such a splitting). Once endowed with the decomposition (\ref{int-convol-intro}), the time-discretization is naturally obtained by keeping only the main term $P$ between two successive times of the partition:
\begin{equation}\label{inter-scheme-intro}
y^M_0=\psi \quad , \quad y^M_{t_{k+1}}=S_{t_{k+1}-t_k} y^M_{t_k}+P_{t_{k+1}t_k},
\end{equation}
with for instance $t_k=t_k^M=k/M$. The reasoning can here be compared with the recent approach by Jentzen and Kloeden for the treatment of a Wiener noise (see \cite{jen-klo-1,jen-klo-2,jen-klo-win}): in order to deduce efficient approximation schemes, the two authors lean on a Taylor expansion of the Wiener solution, which indeed fits the pattern given by (\ref{int-convol-intro}).
\item Then, in contrast to the standard rough systems, an additional step has to be performed in this infinite-dimensional context, so as to retrieve a practically-implementable algorithm. In brief, it consists in projecting the (intermediate) scheme (\ref{inter-scheme-intro}) onto increasing finite-dimensional subspaces of $L^2(0,1)$. We shall carefully examine how to combine this projection with the rough paths machinery (see Subsections \ref{subsec:disc-spa-you} and \ref{subsec:disc-spa-rou}).
\end{list}

\

Let us now present the main results of the paper. To do so, let us be first a little bit more specific about the operator $A$ that we will consider in our study:

\

\textbf{Hypothesis:} \emph{Throughout the paper, we assume that $A$ is a Sturm-Liouville operator with Dirichlet boundary conditions on $(0,1)$ that can be written as $A=\partial_\xi(a \cdot \partial_\xi)+c$, where $c:[0,1] \to \R$ is a continuous function and $a:[0,1] \to \R$ is a continuously differentiable function satisfying $a(\xi) \geq \al$, for some strictly positive constant $\al$}.

\

These conditions ensure in particular the existence of an orthonormal basis $(e_n)$ of eigenfunctions of $A$, and we denote by $(\la_n)$ the sequence of associated eigenvalues (remember that $\la_n \stackrel{n \to \infty}{\longrightarrow} \infty$). The discretization procedure for (\ref{equa-strong}) will highly depend on the Hölder coefficient $\ga$ of $x$: as one might expect, the smaller $\ga$ (i.e., the rougher $x$), the more sophisticated the scheme. In fact, as in the standard rough paths theory, we shall separately deal with the two cases $\ga >1/2$ and $\ga \in (1/3,1/2]$, which will receive distinct treatments. In the following statements, we denote by $\cb_\ka$ ($\ka \geq 0$) the fractional Sobolev spaces associated with $A$ (see Subsection \ref{section-cadre}).

\begin{theorem}[Young case]\label{main-result-young}
Suppose that $\ga \in (\frac{1}{2},1)$ and that Assumptions \textbf{(X1)$_\ga$} and \textbf{(F)$_2$} (see Subsection \ref{subsec:gen-assump}) are both satisfied. Fix $\ga'\in (\max(1-\ga,\frac{\ga}{2}),\frac{1}{2})$ and suppose in addition that $\psi \in \cb_{\ga'}$. Then there exists a function $C:(\R^+)^2 \to \R^+$ bounded on bounded sets such that if $y$ is the mild solution of (\ref{equa-mild}) with initial condition $\psi$ and $y^{M,N}$ is the path generated by the Euler scheme (\ref{euler-scheme}), one has 
\begin{multline}\label{result-schema-young}
\sup_{k\in \{0,\ldots,M\}} \norm{y_{t^M_k}-y^{M,N}_{t^M_k}}_{\cb_{\ga'}}+\sup_{l<k \in \{0,\ldots,M\}} \frac{\norm{(y-y^{M,N})_{t_k^M}-(y-y^{M,N})_{t_l^M}}_\cb}{| t_k^M-t_l^M|^{\ga'}} \\
 \leq C\lp \norm{\psi}_{\cb_{\ga'}},\norm{x}_\ga \rp \lcl \lc \norm{x-x^M}_\ga+\frac{1}{M^{\ga+\ga'-1}}\rc+ \lc \norm{\psi-P_N \psi}_{\cb_{\ga'}}+\frac{1}{\la_N^{\ga-\ga'}}\rc \rcl.
\end{multline}
\end{theorem}

\begin{theorem}[Rough case]\label{theo-rough}
Suppose that $\ga \in (\frac{1}{3},\frac{1}{2}]$ and that Assumptions \textbf{(X2)$_\ga$} and \textbf{(F)$_3$} (see Subsection \ref{subsec:gen-assump}) are both satisfied. Fix $\ga'\in (1-\ga,2\ga]$ and suppose in addition that $\psi \in \cb_{\ga'}$. Then for every parameters $\be,\eta$ satisfying
\begin{equation}\label{cond-para}
0< \be < \inf \lp \ga+\ga'-1, \ga-\ga'+\frac{1}{2} \rp \quad , \quad 0<\eta < \ga+\ga'-1,
\end{equation}
there exists a function $C=C_{\be,\eta}:(\R^+)^2 \to \R^+$ bounded on bounded sets such that if $y$ is the mild solution of (\ref{equa-mild}) in $\cb_{\ga'}$ with initial condition $\psi$ and $y^{M,N}$ is the path generated by the Milstein scheme (\ref{milstein-scheme}), one has
\begin{multline}\label{main-cont-rough}
\sup_{k\in \{0,\ldots,2^M\}} \norm{y_{t^M_k}-y^{M,N}_{t^M_k}}_{\cb_{\ga'}}+\sup_{l<k \in \{0,\ldots,2^M\}} \frac{\norm{(y-y^{M,N})_{t_k^M}-(y-y^{M,N})_{t_l^M}}_\cb}{| t_k^M-t_l^M|^{\ga}}\\
 \leq C\lp \norm{\psi}_{\cb_{\ga'}},\norm{\xrgh}_\ga \rp \lcl \lc \norm{\xrgh-\xrgh^{2^M}}_\ga+\frac{1}{\lp 2^M \rp^\be}\rc +\lc \norm{\psi-P_N \psi}_{\cb_{\ga'}}+\frac{1}{\la_N^{\eta}} \rc \rcl.
\end{multline}
\end{theorem}

Thanks to the bounds exhibited in \cite{DNT} (and recalled in Proposition \ref{bound-rough-fbm}), we can immediately apply Theorems \ref{main-result-young} and \ref{theo-rough} to the fBm situation so as to retrieve almost sure convergence results. Let us focus for instance on the \emph{rough} case:

\begin{corollary}\label{cor-intro-fbm}
Suppose that $x=B$ is a $m$-dimensional fBm with Hurst index $H\in (1/3,1/2]$. Fix $\ga \in (1/3,H)$, $\ga'\in (1-\ga,2\ga]$ and suppose in addition that $\psi$ is infinitely differentiable on $[0,1]$. Denote by $Y$ the (a.s.) mild solution of (\ref{equa-strong}) and by $Y^{M,N}$ the process generated by the Milstein scheme (\ref{milstein-scheme}). Then for every parameters $\be,\eta$ satisfying (\ref{cond-para}), one has
\begin{multline}\label{main-cont-fBm}
\sup_{k\in \{0,\ldots,2^M\}} \norm{Y_{t^M_k}-Y^{M,N}_{t^M_k}}_{\cb_{\ga'}}+\sup_{l<k \in \{0,\ldots,2^M\}} \frac{\norm{(Y-Y^{M,N})_{t_k^M}-(Y-Y^{M,N})_{t_l^M}}_\cb}{| t_k^M-t_l^M|^{\ga}}\\
 \leq C_{\psi,B} \lcl\frac{\sqrt{M}}{(2^M)^{H-\ga}}+\frac{1}{\lp 2^M \rp^\be}+\frac{1}{\la_N^{\eta}} \rcl,
\end{multline}
where $C_{\psi,B}$ is an almost surely finite random variable.
\end{corollary}

\

As far as we are aware, this is the first occurence of approximation schemes for a nonlinear PDE involving a fractional noise, for both the Young case and the rough case. These results provide us with a new evidence of the efficiency of the rough paths approach in the field of numerical integration. Let us now make a few comments about the above statements.

\

\begin{remark}
As we shall see in Section \ref{sec-rough-case}, the use of dyadic intervals for the Milstein scheme (\ref{milstein-scheme}) is justified by the need of a decreasing sequence of partitions in the patching argument of Proposition \ref{prop:contr-j-m-n}. However, our convergence result can probably be extended to any sequence of partitions whose meshes tend to $0$, at the price of more intricate local considerations in the proof of the latter proposition.
\end{remark}

\begin{remark}
The bound (\ref{result-schema-young}) (resp. (\ref{main-cont-rough})) is shaped according to our two-step reasoning. Indeed, the first bracket $[ \norm{x-x^M}_\ga+\frac{1}{M^{\ga+\ga'-1}}]$ (resp. $[ \norm{\xrgh-\xrgh^{2^M}}_\ga+\frac{1}{\lp 2^M \rp^\be}]$) corresponds to the error of approximation due to the time discretization of (\ref{equa-strong}). It can be compared with the bound exhibited in \cite{DNT} for standard rough systems. Then the second bracket $[ \norm{\psi-P_N \psi}_{\cb_{\ga'}}+\frac{1}{\la_N^{\ga-\ga'}}]$ (resp. $[ \norm{\psi-P_N \psi}_{\cb_{\ga'}}+\frac{1}{\la_N^{\eta}} ]$) is the consequence of the projection of the (time-discretized) equation onto the finite-dimensional space $V_N:= \text{Span}\lcl e_n, \ 1\leq n \leq N \rcl$. 
\end{remark}

\begin{remark}
The use of a $\ga$-Hölder norm (with the constraint $\ga >1/3$) in the convergence result (\ref{main-cont-fBm}) is natural in a rough paths setting and is directly linked to the continuity statements for the Itô map associated with (\ref{equa-strong}) (see Theorems \ref{theo-young-abs} and \ref{theo-exi-rough}). Nevertheless, this is not the most common topology for measuring the error of approximation to stochastic PDE, and a more standard criterion would be the supremum norm in $L^2(0,1)$, i.e., $\sup_{k\in \{0,\ldots , 2^M\}} \norm{Y-Y^{M,N}}_{\cb}$. As a particular consequence, it seems difficult to decide by means of numerical observations whether the convergence rate (\ref{main-cont-fBm}) is optimal, owing to the high simulation cost of Hölder norms (not to mention the intricate treatment of the three parameters $H,\ga,\ga'$). In the context of standard rough equations, some progresses have recently been made by Friz and Riedel (\cite{friz-riedel}) concerning (optimal) convergence rates with respect to the supremum norm. We hope that their strategy can be adapted to the rough PDE setting and we plan to address this issue in a further publication.
\end{remark}

\begin{remark}
By keeping track of the constants exhibited at each step of the reasoning, one soon realizes that the almost sure estimate (\ref{main-cont-fBm}) cannot be turned into an $L^1$ estimate, i.e., the random variable $C_{\psi,B}$ is not integrable (see for instance the intermediate bound (\ref{ineq:contr-uni}) for the path $Y^{M,N}$). Note that such average estimates remain an open problem as soon as the Hurst index is strictly smaller than $1/2$, even for rough standard systems (see \cite{DNT}).
\end{remark}

\begin{remark}
According to \cite[Theorem 3.10]{RHE} and \cite[Theorem 2.11]{RHE-glo}, if one wants to interpret (\ref{equa-strong}) in its multidimensional (mild) version, i.e., for a $n$-dimensional Sturm-Liouville operator on $(0,1)^n$, then one must turn to $L^p$-spaces with $p>n$. In particular, one must leave a Hilbert background as soon as $n\geq 2$. This is why we have preferred to stick to the one-dimensional case (and so $p=2$), for which spectral properties of the operator are well-known and projection on finite-dimensional spaces is easily available. Nevertheless, it may be possible to adapt the space-discretization procedure to $L^p$-spaces by introducing e.g. wavelet bases, as in \cite{haus-banach}. It would then be necessary to use generalizations of the basic properties (\ref{projection})-(\ref{pp-2}), and this would also suppose to cope with the two additional parameters $p$ and $n$ throughout the reasoning.
\end{remark}

\begin{remark}\label{extension-rougher}
Our guess is that the strategy when $\ga \in (1/3,1/2]$ (Section \ref{sec-rough-case}) could be adapted to the case where $\ga \in (1/4,1/3]$, at the price of more intricate Taylor expansions and by resorting to a third-order scheme. As reported in \cite[Subsection 2.6]{RHE}, the situation where both $\ga < 1/4$ and $x$ generates a $k$-rough path with $k\geq 4$, is likely to raise additional issues as far as the space parameter $\ga'$ is concerned (when $x$ is a fBm, these hypotheses cannot cover the case where $H<1/4$ anyway, see \cite{CQ}). For the time being, it seems that the only approach able to deal with the condition $\ga <1/4$ in (\ref{equa-strong}) is the BSDE/viscosity-strategy initiated in \cite{diehl-friz}. At this point, we cannot guarantee that the latter (theoretical) treatment remains consistent with our space-time discretization methods, though.
\end{remark}

\

The paper is organized as follows. In Section \ref{settings-main-results}, we elaborate on the assumptions in order throughout our study and we introduce the two approximation schemes under consideration. Section \ref{sec-young-case} is devoted to the proof of Theorem \ref{main-result-young}. Only developments of order $1$ will be involved in this section, so that the scheme can be seen as an adapted version of the usual Euler scheme. In Section \ref{sec-rough-case}, we will handle the scheme associated with Theorem \ref{theo-rough} and which requires developments of order $2$, similarly to the well-known Milstein approximation for standard differential systems. Finally, Appendix A puts together some technical proofs that have been postponed for the sake of clarity, while Appendix B gives an insight into possible implementations of the algorithm in the fBm case.

\section{Settings and schemes}\label{settings-main-results}

\subsection{Assumptions}\label{subsec:gen-assump}
As in \cite{RHE-glo,RHE,GT}, we are interested in the mild formulation of the equation, namely
\begin{equation}\label{equa-mild}
y_t=S_t\psi+\int_0^t S_{t-u} \, dx^i_u \, f_i(y_u) \quad , \quad t\in [0,1] \ , \ \psi \in \cb,
\end{equation}
where $S$ stands for the semigroup generated by $A$. When $x$ is a piecewise differentiable path, the integral involved in (\ref{equa-mild}) is naturally understood as 
$$\int_0^t S_{t-u} \, dx^i_u \, f_i(y_u)=\int_0^t S_{t-u} \, (du \, (x^i)'_u ) \, f_i(y_u).$$
In this context, interpreting the rough version of (\ref{equa-mild}) means extending the ordinary solution $y$ as $x$ tends (in a sense to be precised) to a $\ga$-Hölder path with $\ga <1$ (here $\ga >1/3$, see Remark \ref{extension-rougher}). Let us introduce the topologies that must come into the picture during this extension procedure. First, for any subinterval $I\subset [0,1]$ and any Banach space $V$, we denote by $\cac_1^\ka(I;V)$ ($\ka \in (0,1)$) the set of $\ka$-Hölder paths, endowed with the seminorm
$$\cn[x;\cac_1^\ka(I;V)]:=\sup_{s<t\in I}\frac{|x_t-x_s|_V}{\lln t-s \rrn^\ka}.$$
One also considers the sets $\cac_2^\ka(I;V)$ ($\ka \geq 0$) of two-parameter paths $z$ on $I^2$ (with values in $V$) which satisfy
$$\cn[z;\cac_2^\ka(I;V)]:=\sup_{s<t \in I}\frac{|z_{ts}|_V}{\lln t-s \rrn^\ka}.$$

\smallskip

Now, depending on the Hölder-regularity $\ga$ of $x$, we will be led to assume that one of the two following assumptions is satisfied.

\

\textbf{Assumption (X1)$_\ga$:} $x:[0,1] \to \R^m$ is both a $\ga$-Hölder path and a geometric $1$-rough path. In other words, there exists a sequence of piecewise differentiable path $(x^M)$ such that
$$\norm{x-x^M}_\ga:=\cn[x-x^M;\cac_1^\ga([0,1];\R^m)] \stackrel{M\to \infty}{\longrightarrow} 0.$$

\

\textbf{Assumption (X2)$_\ga$:} $x:[0,1] \to \R^m$ is a $\ga$-Hölder path which gives rise to a geometric $2$-rough path. In other words, there exists a sequence of piecewise differentiable path $(x^M)$ such that $\cn[x-x^M;\cac_1^\ga([0,1];\R^m)] \stackrel{M\to \infty}{\longrightarrow} 0$ and the sequence $(\xrgh^{\bt,M})$ of Lévy areas associated with $(x^M)$, i.e.,
$$\xrgh^{\bt,M,ij}_{ts}:=\int_s^t dx^{M,i}_u \, (x^{M,j}_u-x^{M,j}_s)  , \quad i,j=1,\ldots,m, \quad s<t\in [0,1],$$
converges in $\cac_2^{2\ga}([0,1];\R^{m,m})$ to an element $\xrgh^\bt$.
In brief,
$$\norm{\xrgh-\xrgh^{2^M}}_\ga:=\cn[x-x^M;\cac_1^\ga([0,1];\R^m)] +\cn[\xrgh^{\bt,M}-\xrgh^\bt;\cac_2^{2\ga}([0,1];\R^{m,m})]\stackrel{M\to \infty}{\longrightarrow} 0.$$

\

\textbf{Example:} As pointed out in the introduction, the main process that we have in mind in this paper is the ($m$-dimensional) fractional Brownian motion $x=B^H$ with Hurst index $H>1/3$. It has been indeed proved in \cite{CQ} that this process satisfies Assumption \textbf{(X2)$_\ga$} (and accordingly Assumption \textbf{(X1)$_\ga$}) for any $1/3<\ga <H$, when taking for $x^M$ the linear interpolation of $x$ with uniform mesh $\frac{1}{M}$, i.e.,
\begin{equation}\label{lin-inter}
t_k=t_k^M:=\frac{k}{M} \quad , \quad x^M_t:=x_{t_k}+M\cdot (t-t_k) \cdot (x_{t_{k+1}}-x_{t_k}) \quad \text{if} \ t\in [t_k,t_{k+1}).
\end{equation}
To be more specific, the following bound has been proved in \cite{DNT}, and it allows us to derive Corollary \ref{cor-intro-fbm} from Theorem \ref{theo-rough}.
\begin{proposition}\label{bound-rough-fbm}
Let $x=B$ be a $m$-dimensional fBm with Hurst index $H>1/3$, and let $B^M$ be its linear interpolation with mesh $1/M$. Then, for any $1/3<\ga <H$, there exists an almost surely finite random variable $C_\ga$ such that 
$$\norm{\xrgh-\xrgh^M}_\ga  \leq C_\ga  \sqrt{\log M} \cdot M^{\ga-H}.$$ 
\end{proposition}

\smallskip

\noindent
Note that Condition \textbf{(X1)$_\ga$} or \textbf{(X2)$_\ga$} is actually fulfilled by a larger class of Gaussian processes, as reported in \cite{FVbook}.

\

As far as the regularity of the vector field $f$ is concerned, it will be governed by one of the following conditions ($k$ is a parameter in $\N$).

\

\textbf{Assumption (F)$_k$:} for every $i\in \{1,\ldots,m\}$, $f_i$ belongs to the space $\cac^{k,\textbf{b}}(\R;\R)$ of $k$-time differentiable functions, bounded, with bounded derivatives.

\subsection{Schemes}

Remember that we have fixed an orthonormal basis $(e_n)$ made of eigenfunctions of $A$. For any $N\in \N^\ast$, we denote by $P_N$ the projection operator onto the finite-dimensional subspace $V_N:= \text{Span}\lcl e_n, \ 1\leq n \leq N \rcl$. 

\smallskip

In order to introduce the two schemes associated with Theorems \ref{main-result-young} and \ref{theo-rough}, let us first define, for any piecewise differentiable path $\xti:[0,1]\to \R^m$, the following operator-valued paths: for $i,j=1,\ldots,m$ and $s<t\in [0,1]$,
\begin{equation}\label{def-reg-x}
X^{\xti,i}_{ts}:=\int_s^t S_{t-u} \, d\xti^i_u \quad , \quad X^{\xti\xti,ij}_{ts}:=\int_s^t S_{t-u} \, d\xti^i_u \, (\xti^j_u-\xti^j_s),
\end{equation}

\smallskip

\noindent
We suppose in addition that either Assumption \textbf{(X1)$_\ga$} or Assumption \textbf{(X2)$_\ga$} is satisfied for some parameter $\ga\in (0,1)$ and some fixed regularizing sequence $(x^M)$, and that Assumption $\textbf{(F)$_1$}$ holds true (in particular, $f_i'$ is well-defined). 

\

\textbf{Euler scheme:} For $M,N \in \N$, $y^{M,N}_0=P_N \psi$ and
\begin{equation}\label{euler-scheme}
y^{M,N}_{t_{k+1}} =S_{t_{k+1}-t_k} y^{M,N}_{t_k}+ X^{x^M,i}_{t_{k+1}t_k} P_N f_i(y^{M,N}_{t_k}),
\end{equation}
where, for every $k\in \{0,\ldots,M\}$, $t_k=t_k^M:=\frac{k}{M}$.

\

\textbf{Milstein scheme:} For $M,N \in \N$, $y^{M,N}_0=P_N \psi$ and
\begin{equation}\label{milstein-scheme}
y^{M,N}_{t_{k+1}} =S_{t_{k+1}-t_k} y^{M,N}_{t_k}+ X^{x^{2^M},i}_{t_{k+1}t_k} P_N f_i(y^{M,N}_{t_k})+X^{x^{2^M} x^{2^M},ij}_{t_{k+1}t_k} P_N\lp f_i'(y^{M,N}_{t_k}) \cdot (P_N f_j(y^{M,N}_{t_k})) \rp,
\end{equation}
where, for every $k\in \{ 0,\ldots,2^M\}$, $t_k=t_k^M:=\frac{k}{2^M}$, and the notation $\phi \cdot \psi$ stands for the pointwise product of functions, i.e., $(\phi \cdot \psi)(\xi):=\phi(\xi) \psi(\xi)$.

\

\begin{remark}
The two schemes are of course named after the classical algorithms for standard stochastic differential equations (see \cite{klo-pla}). 
\end{remark}

\begin{remark}
When $x^M$ is the linear interpolation of $x$ given by (\ref{lin-inter}), the two sequences of operators $X^{x^M,i}_{t_{k+1}t_k},X^{x^Mx^M,ij}_{t_{k+1}t_k}$ reduce to
\begin{equation}
X^{x^M,i}_{t_{k+1}t_k}=M \cdot (x^i_{t_{k+1}}-x^i_{t_k}) \cdot \int_{t_k}^{t_{k+1}} S_{t_{k+1}-u} \, du
\end{equation}
\begin{equation}
X^{x^M x^M,ij}_{t_{k+1}t_k}=M^2 \cdot (x^i_{t_{k+1}}-x^i_{t_k}) \cdot (x^j_{t_{k+1}}-x^j_{t_k}) \cdot \int_{t_k}^{t_{k+1}} S_{t_{k+1}-u} \, du \, (u-t_k).
\end{equation}
Consequently, in this case, the computations of Formulas (\ref{euler-scheme}) and (\ref{milstein-scheme}) only require the a priori knowledge of the successive increments $x_{t_{k+1}}-x_{t_k}$, which makes the implementation of the algorithms very easy, as we will see in Appendix B for the fBm case.
\end{remark}

\subsection{Fractional Sobolev spaces}\label{section-cadre}

In order to compensate for the lack of time regularity in the integral involved in (\ref{equa-mild}), we shall take advantage of the space regularity of the initial condition (and then the solution itself). This is a classical strategy for (Wiener) SPDE (see \cite{daprato-zab}), which traditionally appeals to the so-called fractional Sobolev spaces.

\begin{notation}
For any $\ka \geq 0$, we denote by $\cb_\ka$ the fractional Sobolev space associated with $(-A)^\ka$ and characterized by
\begin{equation}\label{sobo-char}
\cb_\ka =\{ y \in L^2(0,1): \ \sum_{n=1}^\infty \la_n^{2\ka} (y^n)^2 < \infty \},
\end{equation}
where the $(y^n)$ are the components of $y$ in the basis $(e_n)$. This space is naturally equipped with the norm
\begin{equation}\label{norme-b-ka}
\norm{y}_{\cb_\ka}^2=\norm{(-A)^\ka y}_\cb^2=\sum_{n=1}^\infty \la_n^{2\ka}(y^n)^2,
\end{equation}
and we extend the definition of $\cb_\ka$ to any $\ka <0$ through the characterization formula (\ref{sobo-char}).
\end{notation}

As reported in \cite[Section 2.1]{RHE-glo}, the fractional Sobolev spaces $\cb_\ka$ coincide here with the classical Triebel-Lizorkin spaces $F^{2\ka}_{2,2}$, which are described and extensively studied in \cite{run-sick} (for instance). Let us draw up a list of some of the properties of these spaces, which will be used throughout our reasoning.

\smallskip

\begin{list}{\labelitemi}{\leftmargin=1em \itemsep=1em}
\item \textit{Sobolev inclusions}: If $\ka >1/4$, $\cb_\ka$ is a Banach algebra with respect to pointwise multiplication, i.e., $\| \vp \cdot \psi\|_{\cb_\ka} \leq \| \vp \|_{\cb_\ka} \| \psi\|_{\cb_\ka}$ for all $\vp, \psi \in \cb_\ka$. Moreover, it is continuously included in the space $L^\infty([0,1])$ of bounded functions on $[0,1]$.
\item \textit{Projection}: For all $0\leq \ka < \al$ and for any $\vp \in \cb_\al$, 
\begin{equation}\label{projection}
\norm{\vp-P_N \vp}_{\cb_\ka} \leq \la_N^{-(\al-\ka)} \norm{\vp}_{\cb_\al}.
\end{equation}
\item \textit{Contraction}: For any $\ka \geq 0$, $S$ is a contraction operator on $\cb_\ka$.
\item \textit{Regularization}: For any $t>0$ and for all $-\infty < \ka < \al < \infty$, $S_t$ sends $\cb_\ka$ into $\cb_\al$ and one has
\begin{equation}\label{regularisation-semigroupe}
\norm{S_t \vp}_{\cb_\al} \leq c_{\al,\ka} t^{-(\al-\ka)} \norm{\vp}_{\cb_\ka}.
\end{equation} 
\item \textit{Hölder regularity}: For all $t>0,\al>0$ and for any $\vp \in \cb_\al$, one has
\begin{equation}\label{hol-reg}
\norm{S_t \vp-\vp}_{\cb} \leq c_\al t^\al \norm{\vp}_{\cb_\al} \quad , \quad \norm{A S_t \vp}_{\cb} \leq c_\al t^{-1+\al} \norm{\vp}_{\cb_\al}.
\end{equation}
\item \textit{Composition} (see \cite{sickel}): if $\ka \in [0,1/2]$ and $f\in \cac^{1,\textbf{b}}(\R;\R)$ (see Assumption \textbf{(F)$_k$} for the latter notation), then
\begin{equation}\label{contr-f-y-sobo}
\norm{f(\vp)}_{\cb_\ka} \leq c_{\ka,f} \lcl 1+\norm{\vp}_{\cb_\ka} \rcl,
\end{equation}
while if $\ka \in (1/2,1)$ and $f\in \cac^{2,\textbf{b}}(\R;\R)$, one has
\begin{equation}
\norm{f(\vp)}_{\cb_\ka} \leq c_{\ka,f} \lcl 1+\norm{\vp}_{\cb_\ka}^2 \rcl,
\end{equation}
where, in both cases, $f(\vp)$ is understood in the sense of composition, i.e., $f(\vp)(\xi):=f(\vp(\xi))$.
\item \textit{Pointwise product} (see \cite[Sections 4.6.4 and 4.6.1]{run-sick}): if $\ka \in [0,1/2]$ and $\vp,\psi \in \cb_\ka \cap L^\infty([0,1])$, then
\begin{equation}\label{pp-1}
\norm{\vp \cdot \psi}_{\cb_\ka} \leq c_\ka \lcl \norm{\vp}_{L^\infty}\norm{\psi}_{\cb_\ka}+\norm{\vp}_{\cb_\ka}\norm{\psi}_{L^\infty} \rcl,
\end{equation}
while if $\vp\in \cb_{-\ka}$, $\psi\in \cb_\al$, with $\ka\geq 0$ and $\al >\max(\ka,\frac{1}{4})$, one has
\begin{equation}\label{pp-2}
\norm{\vp \cdot \psi}_{\cb_{-\ka}} \leq c_{\ka,\al} \norm{\vp}_{\cb_{-\ka}} \norm{\psi}_{\cb_\al}.
\end{equation}
\end{list}

\subsection{Tools of algebraic integration}\label{subsec:tools}

With the above properties in hand, let us recall that the rough paths treatment of (\ref{equa-mild}) (as it is developed in \cite{GT,RHE,RHE-glo}) is based on the controlled expansion of the convolutional integral
\begin{equation}\label{conv-in} 
\int_s^t S_{t-u} \, dx^i_u \, f_i(y_u).
\end{equation}
In order to express this control with the highest accuracy, we provide ourselves with a few tools inspired by the algebraic integration theory for standard systems (see \cite{gubi}).

\smallskip

\noindent
\textbf{Notation.} For $k\in \{1,2,3\}$ and for any interval $I\subset [0,1]$, set
$$\cs_k(I):=\{(t_1,\ldots,t_k) \in I^k: \ t_1 \geq \ldots \geq t_k \}.$$
Then for all paths $y:I \to \cb$ and $z:\cs_2(I) \to \cb$,  we define, if $s\leq u\leq t \in I$,
\begin{equation}\label{incr-std}
(\der y)_{ts}:=y_t-y_s \quad , \quad (\delha y)_{ts}:=(\der y)_{ts}-a_{ts}y_s,
\end{equation}
\begin{equation}\label{incr-twis}
(\delha z)_{tus}:=z_{ts}-z_{tu}-S_{t-u}z_{us},
\end{equation}
where $a_{ts}:=S_{t-s}-\Id$.

\smallskip

\noindent
To give an idea on how these operators arise from Equation (\ref{equa-mild}), let us observe that due to the additivity property $S_{t+t'}=S_t S_{t'}$, the variations of the (ordinary) solution $y$ are governed by the equation
$$(\der y)_{ts}=\int_s^t S_{t-u} \, dx^i_u \, f_i(y_u)+a_{ts}\int_0^s S_{s-u} \, dx^i_u \, f_i(y_u)=\int_s^t S_{t-u} \, dx^i_u \, f_i(y_u)+a_{ts}y_s,$$
and (\ref{equa-mild}) can thus be equivalently written as
\begin{equation}\label{eq-base}
y_0=\psi \quad , \quad (\delha y)_{ts}=\int_s^t S_{t-u} \, dx^i_u \, f_i(y_u).
\end{equation}
In this convolutional context, let us also observe the following elementary properties, that we label for a further use:
\begin{proposition}\label{prop:tele}
Let $y:[0,1] \to \cb$, $z:\cs_2([0,1]) \to \cb$, and let $x:[0,1] \to \R$ be a differentiable path. Then it holds:
\begin{itemize}
\item Telescopic sum: $[ \delha (\delha y)]_{tus}=0$ and $(\delha y)_{ts}=\sum_{i=0}^{n-1} S_{t-t_{i+1}}(\delha y)_{t_{i+1}t_i}$ for any partition $\{s=t_0 <t_1 <\ldots <t_n=t\}$ of an interval $[s,t]$ of $[0,1]$.
\item Chasles relation: if $\cj_{ts}:=\int_s^t S_{t-u} \, dx_u \, y_u$, then $\delha \cj=0$.
\item Cohomology: if $\delha z=0$, then there exists $h:[0,1] \to \cb$ such that $\delha h=z$.
\end{itemize}
\end{proposition}

\smallskip

\noindent
On top of these algebraic considerations, if one wants to measure the regularity of the terms involved in the expansion of $\int_s^t S_{t-u} \, dx^i_u \, f_i(y_u)$, one is led to introduce the following suitable semi-norms, that can be seen as generalizations of the classical Hölder norm: if $y:I \to V$, $z:\cs_2(I) \to V$ and $h:\cs_3(I) \to V$, where $I\subset [0,1]$ and $V$ is any Banach space, we define, for any $\la >0$,
\begin{equation}
\cn[y;\cacha_1^\la(I;V)] :=\sup_{ s<t \in I} \frac{\norm{(\delha y)_{ts}}_V}{\lln t-s \rrn^\la} \quad , \quad \cn[y;\cac_1^0(I;V)]:=\sup_{t\in I}\,  \norm{y_t}_V,
\end{equation}
\begin{equation}
\cn[z;\cac_2^\la(I;V)]:=\sup_{ s <t \in I} \frac{\norm{z_{ts}}_V}{\lln t-s \rrn^\la} \quad , \quad \cn[h;\cac_3^\la(I;V)]:=\sup_{ s <u<t\in I} \frac{\norm{h_{tus}}_V}{\lln t-s \rrn^\la}. 
\end{equation}
Then $\cacha_1^\la(I;V)$ naturally stands for the set of paths $y:I \to V$ such that $\cn[y;\cacha_1^\la(I;V)]<\infty$, and we define $\cac_1^0(I;V)$, $\cac_2^\la(I;V)$ and $\cac_3^\la(I;V)$ along the same line. With this notation, observe for instance that if $y\in \cac_2^\la(I;\cl(V,W))$ and $z\in \cac_2^\be(I;V)$, then the path $h$ defined as $h_{tus}:=y_{tu}z_{us}$ ($s\leq u\leq t$) belongs to $\cac_3^{\la+\be}(I;W)$.

\smallskip

\noindent
When $I=[0,1]$, we will more simply write $\cac_k^\la(V):=\cac_k^\la(I;V)$.

\smallskip

\noindent
The following notational convention also turns out to be useful as soon as products of paths come into play:
\begin{notation}\label{convention-indices}
If $g:\cs_n \to \cl(V,W)$ and $h:\cs_m \to V$, then the product $gh: \cs_{n+m-1} \to W$ is defined by the formula
$$(gh)_{t_1 \ldots t_{m+n-1}}:=g_{t_1 \ldots t_n}h_{t_n \ldots t_{n+m-1}}.$$
\end{notation}

\noindent
With this convention, it is readily checked that if $g: \cs_2 \to \cl(\cb_\ka,\cb_\al)$ and $h:\cs_n \to \cb_\ka$, then $\delha(gh): \cs_{n+1} \to\cb_\al$ is given by
\begin{equation}\label{rel-alg-prod}
\delha (gh)=(\delha g)h-g(\der h).
\end{equation}

\smallskip

\noindent
To end up with this toolbox, let us report one of the cornerstone results of \cite{GT}, namely the existence of (some kind of) an inverse operator for $\delha$, denoted by $\Laha$, and which will play an important role in the reasoning of Section \ref{sec-rough-case}. In brief, let us say that this operator allows us to get both a nice expression and a sharp control for the \emph{smooth} terms, i.e., the terms with Hölder regularity greater than $1$, arising from the expansion of the rough integral $\int_s^t S_{t-u} \, dx^i_u \, f_i(y_u)$ (see e.g. Lemma \ref{lem:dec-smoo}).

\begin{theorem}\label{existence-laha}
Fix an interval $I\subset [0,1]$, a parameter $\ka \geq 0$ and let $\mu >1$. For any $h\in \cac_3^\mu(I;\cb_{\ka})\cap \text{Im} \, \delha$, there exists a unique element $$\Laha h \in \cap_{\al\in [0,\mu)} \cac_2^{\mu-\al}(I;\cb_{\ka+\al})$$ such that $\delha(\Laha h)=h$. Moreover, $\Laha h$ satisfies the following contraction property: for all $\al\in [0,\mu)$,
\begin{equation}\label{contraction-laha}
\cn[\Laha h;\cac_2^{\mu-\al}(I;\cb_{\ka+\al})] \leq c_{\al,\mu} \, \cn[h;\cac_3^\mu(I;\cb_{\ka})].
\end{equation}
\end{theorem}

\section{Young case}\label{sec-young-case}

This section is devoted to the proof of Theorem \ref{main-result-young}. Consequently, we fix from now on the two parameters $\ga \in (\frac{1}{2},1)$ and $\ga'\in (\max(\frac{1}{4},1-\ga),\frac{1}{2})$, as well as the initial condition $\psi \in \cb_{\ga'}$. We also fix the approximating sequence $(x^M)$ of $x$ given by Assumption \textbf{(X1)$_\ga$}.

\smallskip

The first point to elaborate on here is that under both Assumptions \textbf{(X1)$_\ga$} and \textbf{(F)$_2$}, the convolution integral (\ref{conv-in}) can be extended to a path $x\in \cac_1^\ga(\R^m)$ via a first-order expansion. To this end, the strategy is based on the following elementary decomposition.

\begin{lemma}\label{lem:dec-smoo}
If $\xti:[0,1]\to \R$ and $z:[0,1]\to \cb$ are both piecewise continuously differentiable paths, then the following decomposition is in order:
\begin{equation}\label{dec-you-regu}
\int_s^t S_{t-u} \, d\xti^i_u \, z_u=X^{\xti,i}_{ts} z_s+\Laha_{ts}(X^{\xti,i} \der z),
\end{equation} 
where $X^{\xti,i}$ is the operator-valued path defined by (\ref{def-reg-x}).
\end{lemma}

\begin{proof}
Set $J_{ts}:=\int_s^t S_{t-u} \, d\xti^i_u \, z_u-X^{\xti,i}_{ts} z_s=\int_s^t S_{t-u} \, d\xti^i_u \, (\der z)_{us}$. Owing to the regularity of $\xti$ and $z$, it is clear that $J\in \cac_2^2(\cb)$. Moreover, one has $\delha (J-\Laha (X^{\xti} \der z))=X^{\xti} \der z-X^{\xti} \der z=0$, so $J-\Laha(X^{\xti} \der z)=\delha h$ for some path $h\in \cac_2(\cb)$. According to Theorem \ref{existence-laha}, we know that $\Laha(X^{\xti} \der z) \in \cac_2^2(\cb)$ and hence $\delha h \in \cac_2^2(\cb)$, which easily entails $\delha h=0$ (use the telescopic-sum property of Proposition \ref{prop:tele}). 
\end{proof}

One can then rely on the following extension result for $X^x$:

\begin{lemma}[\cite{RHE}, Proposition 6.3]\label{lem-defi-x-x}
Under Assumption \textbf{(X1)$_\ga$}, the sequence of operator-valued paths 
$$X^{x^M,i}_{ts}:=\int_s^t S_{t-u} \, dx^{M,i}_u$$
converges to an element $X^{x,i}$ with respect to the topology of the spaces $\cac_2^{\ga-\la}(\cl(\cb_{\ka},\cb_{\ka+\la}))$ ($\la \in [0,\ga),\ka\in \R$) and $\cn[X^{x,i};\cac_2^{\ga-\la}(\cl(\cb_\ka,\cb_{\ka+\la}))] \leq c_{\ka,\la} \norm{x}_\ga$, as well as
\begin{equation}\label{cont-cons-x-x}
\cn[X^{x,i}-X^{x^M,i};\cac_2^{\ga-\la}(\cl(\cb_\ka,\cb_{\ka+\la}))] \leq c_{\ka,\la} \norm{x-x^M}_\ga.
\end{equation}
Moreover, $X^{x,i}$ commutes with the projection $P_N$ and it satisfies the algebraic relation $\delha X^{x,i}=0$. 
\end{lemma}

\begin{remark}
The operator $X^{x,i}_{ts}$ morally behaves like $S_{t-s}(\der x^i)_{ts}$ as far as space-time regularity is concerned. The above control $\cn[X^{x,i};\cac_2^{\ga-\la}(\cl(\cb_\ka,\cb_{\ka+\la}))] \leq c_{\ka,\la} \norm{x}_\ga$ can thus be seen as a consequence of the regularizing property (\ref{regularisation-semigroupe}), since for any $\vp \in \cb_\ka$, one has
$$\norm{S_{t-s}(\vp) (\der x^i)_{ts}}_{\cb_{\ka+\la}} \leq c \lln t-s \rrn^{-\la} |(\der x^i)_{ts}| \norm{\vp}_{\cb_\ka}\leq c \lln t-s \rrn^{\ga-\la} \norm{x}_\ga \norm{\vp }_{\cb_\ka}.$$ 
\end{remark}

\begin{remark}
Through the continuity result (\ref{cont-cons-x-x}), one can see that the path $X^x$ only depends on $x$ and not on the particular approximating sequence $x^M$. This comment also holds for the forthcoming Lemma \ref{lem-prol-x-x-x}.
\end{remark}

Once endowed with $X^x$, it is readily checked that the right-hand-side of (\ref{dec-you-regu}) can also be extended to a class of non-differentiable paths $z$, which provides us with the expected interpretation:

\begin{proposition}[\cite{RHE}, Proposition 3.9]\label{defi-int-young}
Under Assumption \textbf{(X1)$_\ga$}, we define, for any path $z=(z^1,  \ldots ,z^m)$ such that $z^i \in \cac_1^0(\cb_\ka) \cap \cac_1^\ka(\cb)$ with $\ga +\ka >1$, the integral
\begin{equation}\label{decompo-int-young}
\cj_{ts}(\hat{d}x \, z):=X^{x,i}_{ts}z^i_s+\Laha_{ts}\lp X^{x,i} \der z^i\rp.
\end{equation}
Then:
\begin{itemize}
\item $\cj(\hat{d} x \, z)$ is well-defined via Theorem \ref{existence-laha}. It coincides with the Lebesgue integral $\int_s^t S_{t-u} \, dx^i_u \, z^i_u$ when $x$ is a piecewise differentiable path.
\item The following estimate holds true:
\begin{equation}\label{estimation-cas-young}
\cn[\cj(\hat{d} x \, z);\cac_2^\ga(\cb_{\ka})] \leq c \norm{x}_{\ga} \lcl \cn[ z;\cac_1^0(\cb_{\ka})]+\cn[z;\cac_1^\ka(\cb)] \rcl.
\end{equation}
\end{itemize}
\end{proposition}

\

It remains to notice that this result applies in particular to the interpretation of Equation (\ref{equa-mild}) as soon as $y\in \cac_1^0(\cb_{\ga'}) \cap \cacha_1^{\ga'}(\cb_{\ga'})$ and Assumption \textbf{(F)$_1$} is in order. Indeed, thanks to (\ref{contr-f-y-sobo}), one has $f_i(y) \in \cac_1^0(\cb_{\ga'})$, while due to (\ref{hol-reg}) it holds that
\begin{equation}\label{preci:hol}
\cn[f_i(y);\cac_1^{\ga'}(\cb)] \leq \norm{f'}_\infty \cn[y;\cac_1^{\ga'}(\cb)] \leq c \lcl \cn[y;\cac_1^0(\cb_{\ga'})]+\cn[y;\cacha_1^{\ga'}(\cb_{\ga'})] \rcl \ < \infty,
\end{equation}
so $f_i(y) \in \cac_1^0(\cb_{\ga'}) \cap \cac_1^{\ga'}(\cb)$ and we have assumed that $\ga+\ga'>1$.

\subsection{Previous results}

The main result of \cite{RHE} for the Young case is summed up by the following statement:

\begin{theorem}[\cite{RHE}, Theorem 3.10]\label{theo-young-abs}
Under Assumptions \textbf{(X1)$_\ga$} and \textbf{(F)$_2$}, Equation (\ref{equa-mild}) interpreted thanks to Proposition \ref{defi-int-young} admits a unique solution $y$ in $\cacha_1^{\ga'}(\cb_{\ga'})$, and the following estimate holds true:
\begin{equation}\label{contr-sol-young}
\cn[y;\cac_1^{0}(\cb_{\ga'})] +\cn[y;\cacha_1^{\ga'}(\cb_{\ga'})] \leq C\lp \norm{\psi}_{\cb_{\ga'}},\norm{x}_{\ga} \rp,
\end{equation}
for some function $C:(\R^+)^2 \to \R^+$ bounded on bounded sets. Morever, if $y$ (resp. $\yti$) is the solution of (\ref{equa-mild}) associated with a path $x$ (resp. $\xti$) that satisfies Assumption \textbf{(X1)$_\ga$}, with initial condition $\psi$ (resp. $\tilde{\psi}$) in $\cb_{\ga'}$,
\begin{equation}\label{cont-yg}
\cn[y-\yti;\cac_1^0(\cb_{\ga'})]+\cn[y-\yti;\cacha_1^{\ga'}(\cb_{\ga'})] \leq c_{x,\xti,\psi,\tilde{\psi}} \lcl \norm{\psi-\tilde{\psi}}_{\cb_{\ga'}}+\norm{x-\xti}_\ga \rcl,
\end{equation}
with $c_{x,\xti,\psi,\tilde{\psi}}:=C'(\norm{x}_\ga,\norm{\xti}_\ga,\norm{\psi}_{\cb_{\ga'}},\norm{\tilde{\psi}}_{\cb_{\ga'}})$, for some function $C':(\R^+)^4 \to \R^+$ bounded on bounded sets.
\end{theorem}

\begin{remark}\label{rk:change-topo}
It is worth noticing that (\ref{contr-sol-young}) and (\ref{estimation-cas-young}) entails in particular
$$\cn[y;\cacha_1^\ga(\cb_{\ga'})] \leq c_{\psi,x}.$$
Indeed, since $y$ is solution to the system, one has
\bean
\norm{(\delha y)_{ts}}_{\cb_{\ga'}} &\leq & \norm{\cj_{ts}(\hat{d} x \, f(y))}_{\cb_{\ga'}}\\
& \leq &c_x \lln t-s \rrn^\ga \lcl \cn[f(y);\cac_1^0(\cb_{\ga'})]+\cn[f(y);\cac_1^{\ga'}(\cb)]\rcl.
\eean
Then, thanks to (\ref{contr-f-y-sobo}) and (\ref{hol-reg}), it holds that $\cn[f(y);\cac_1^0(\cb_{\ga'})] \leq c \lcl 1+\cn[y;\cac_1^0(\cb_{\ga'})]\rcl$ and as in (\ref{preci:hol}), $\cn[f(y);\cac_1^{\ga'}(\cb)]  \leq c \,  \lcl \cn[y;\cac_1^{0}(\cb_{\ga'})]+\cn[y;\cacha_1^{{\ga'}}(\cb_{\ga'})] \rcl$.
\end{remark}

The continuity result (\ref{cont-yg}) provides us with a control over the discretization of the driving signal $x$. This is the first step towards Theorem \ref{main-result-young}:

\begin{notation}
For any $M\in \N$, we denote by $\yba^M$ the Wong-Zakaï approximation associated with $x^M$ (with the same initial condition $\psi$), or otherwise stated the solution to Equation (\ref{equa-mild}) when $x$ is replaced with its approximation $x^M$. 
\end{notation}

\begin{corollary}\label{cor-y-yba}
With the above notation, there exists a function $C:(\R^+)^2 \to \R^+$ bounded on bounded sets such that, for any $M\in \N$,
\begin{equation}
\cn[y-\yba^M;\cac_1^0(\cb_{\ga'})]+\cn[y-\yba^M;\cacha_1^{\ga'}(\cb_{\ga'})] \leq C(\norm{x}_\ga,\norm{\psi}_{\cb_{\ga'}})\cdot  \norm{x-x^M}_\ga.
\end{equation}
\end{corollary}

\subsection{A uniform control}

The second step of our reasoning consists in controlling the path $y^{M,N}$ generated by (\ref{euler-scheme}) uniformly with respect to $M$ and $N$. To do so, let us first extend $y^{M,N}$ on $[0,1]$ through the formula: if $t\in [t_k,t_{k+1})$,
\begin{equation}\label{ext-0-1}
y^{M,N}_t:=S_{t-t_k}y^{M,N}_{t_k}+X^{x^M,i}_{tt_k} P_N f_i(y^{M,N}_{t_k}).
\end{equation}

Now set 
$$r^{M,N}_{ts}:=\Laha_{ts} \lp X^{x^M,i}\,  \der P_N f_i(y^{M,N}) \rp$$
and observe that one can write, for any $k\in\{0,\ldots,M-1\}$,
\begin{equation}\label{ecrit-y-m-int}
y^{M,N}_{t_{k+1}}=S_{t_{k+1}-t_k}y^{M,N}_{t_k}+\int_{t_k}^{t_{k+1}}S_{t_{k+1}-u} \, dx_u^{i,M} \, P_N f_i(y^{M,N}_u)-r^{M,N}_{t_{k+1}t_k}.
\end{equation}
Extending the expression to all times $s<t$ gives birth to the two following formulas:

\begin{lemma}\label{lem-interm-young}
If $t_p \leq s <t_{p+1} < \ldots <t_q \leq t < t_{q+1}$, then
\begin{equation}\label{decompo-y-m}
(\delha y^{M,N})_{ts}=\int_s^t S_{t-u} \, dx^{i,M}_u \, P_N f_i(y^{M,N}_u)-y^{M,N,\sharp}_{ts},
\end{equation}
with
\begin{equation}
y^{M,N,\sharp}_{ts}:=r^{M,N}_{tt_q}-S_{t-s}r^{M,N}_{st_p}+\sum_{k=p}^{q-1} S_{t-t_{k+1}}r^{M,N}_{t_{k+1}t_k},
\end{equation}
while if $t_p \leq s < t < t_{p+1}$,
\begin{equation}\label{decompo-y-m-bis}
(\delha y^{M,N})_{ts}=X^{x^M,i}_{ts}P_N f_i(y^{M,N}_{t_p}).
\end{equation}

\end{lemma}

\begin{proof}
Formula (\ref{decompo-y-m-bis}) is a straightforward consequence of the relation $\delha X^{x^M,i}=0$. Formula (\ref{decompo-y-m}) follows from the association of (\ref{ecrit-y-m-int}) with the telescopic-sum property contained in Proposition \ref{prop:tele}, which gives here
\begin{eqnarray}\label{decompo-inter}
(\delha y^{M,N})_{ts} &=& \sum_{k=p+1}^{q-1} S_{t-t_{k+1}} (\delha y^{M,N})_{t_{k+1}t_k}+(\delha y^{M,N})_{tt_q}+S_{t-t_{p+1}} (\delha y^{M,N})_{t_{p+1}s}  \nonumber\\
&=& \lc \int_{t_{p+1}}^{t_q} S_{t-u} \, dx^{i,M}_u \, P_N f_i(y^{M,N}_u)-\sum_{k=p+1}^{q-1}S_{t-t_{k+1}} r^{M,N}_{t_{k+1}t_k}\rc \nonumber\\
& & +\lc \int_{t_q}^t S_{t-u} \, dx^{i,M}_u \, P_N f_i(y^{M,N}_u)-r^{M,N}_{tt_q}\rc+S_{t-t_{p+1}} (\delha y^{M,N})_{t_{p+1}s}.
\end{eqnarray}
Then\begin{eqnarray}\label{decompo-inter-de}
(\delha y^{M,N})_{t_{p+1}s} & = &  (\delha y^{M,N})_{t_{p+1}t_p}-S_{t_{p+1}-s}(\delha y^{M,N})_{st_p}\nonumber\\
&=& \lc \int_{t_p}^{t_{p+1}} S_{t_{p+1}-u} \, dx^{i,M}_u \, P_N f_i(y^{M,N}_u)+r^{M,N}_{t_{p+1}t_p} \rc \nonumber\\
& & \hspace{1cm}-S_{t_{p+1}-s} \lc \int_{t_p}^s S_{s-u} \, dx^{i,M}_u \, P_N f_i(y^{M,N}_u)-r^{M,N}_{st_p} \rc,
\end{eqnarray}
and it suffices to inject (\ref{decompo-inter-de}) in (\ref{decompo-inter}) to get (\ref{decompo-y-m}).

\end{proof}

We are going to lean on the two expressions (\ref{decompo-y-m}) and (\ref{decompo-y-m-bis}) in order to establish the expected uniform estimate:
\begin{proposition}\label{prop-unif-y-m-n}
There exists two constants $C_1,C_2>0$ such that for every $M,N \in \N$,
\begin{equation}\label{bound-unif-young}
\cn[y^{M,N};\cac_1^{0}([0,1],\cb_{\ga'})]+\cn[y^{M,N};\cacha_1^{{\ga'}}([0,1],\cb_{\ga'})] \leq C_1 \lcl 1+\norm{\psi}_{\cb_{\ga'}} \rcl \exp \lp C_2 \norm{x}_\ga^{1/(\ga-\ga')}\rp,
\end{equation}
where $y^{M,N}$ is extended on $[0,1]$ through Formula (\ref{ext-0-1}).
\end{proposition}

\begin{proof}
For the sake of conciseness, we use the short notation
$$\cn[y^{M,N};\cacha_1^{0,{\ga'}}(I)]:=\cn[y^{M,N};\cac_1^{0}(I,\cb_{\ga'})]+\cn[y^{M,N};\cacha_1^{{\ga'}}(I,\cb_{\ga'})].$$
We will actually prove the following assertion: there exists a time $T_0=T_0 (\norm{x}_\ga)>0$ and a sequence of radii $R_l=R_l(\norm{x}_\ga,\norm{\psi}_{\cb_{\ga'}})$ such that for any $l$,
\begin{equation}\label{hypo-recur}
\cn[y^{M,N};\cacha_1^{0,{\ga'}}([0,lT_0])] \leq R_l.
\end{equation}
For $l=0$, take $R_0:=\norm{\psi}_{\cb_{\ga'}}$. Now assume that the property holds true for $l$, and let $s,t\in [0,(l+1)T_0]$.

\smallskip

\textit{1$^{\text{st}}$ case}: $s,t \in [lT_0,(l+1)T_0]$.

\smallskip

\textit{1$^{\text{st}}$ subcase}: $t_p \leq s < t_{p+1} < \ldots <t_q \leq t <t_{q+1}$, with $\lln t-s \rrn \geq \frac{1}{M}$. Then, from (\ref{decompo-y-m}),
$$(\delha y^{M,N})_{ts}=\int_s^t S_{t-u} \, dx^{i,M}_u \, P_N f_i(y^{M,N}_u)-y^{M,N,\sharp}_{ts}.$$
Owing to the estimate (\ref{estimation-cas-young}) (applied to $x=x^M$), one easily deduces (see Remark \ref{rk:change-topo})
$$\norm{\int_s^t S_{t-u} \, dx^{i,M}_u \, P_N f_i(y^{M,N}_u)}_{\cb_{\ga'}} \leq c \norm{x}_\ga \lln t-s \rrn^{\ga'} T_0^{\ga-{\ga'}} \lcl 1+\cn[y^{M,N};\cacha_1^{0,{\ga'}}([0,(l+1)T_0])] \rcl.$$
Besides, thanks to the contraction property (\ref{contraction-laha}) of $\Laha$, we get
\begin{eqnarray}
\norm{r^{M,N}_{ts}}_{\cb} &\leq & c \lln t-s \rrn^{\ga+\ga'} \cn[X^{x^M,i} \der P_N(f_i(y^{M,N}));\cac_2^{\ga+\ga'}([0,(l+1)T_0];\cb)] \nonumber\\
&\leq & c \lln t-s \rrn^{\ga+\ga'} \cn[X^{x^M,i};\cac_2^\ga(\cl(\cb,\cb))] \cdot \cn[f_i(y^{M,N});\cac_1^{\ga'}([0,(l+1)T_0];\cb)]\nonumber \\
&\leq & c \norm{x}_\ga \lln t-s \rrn^{\ga+{\ga'}} \cn[y^{M,N};\cacha_1^{0,{\ga'}}([0,(l+1)T_0])] ,\label{est-r-you-1}
\end{eqnarray}
where we have used Lemma \ref{lem-defi-x-x} and (\ref{preci:hol}) to get the last inequality. Using the contraction property (\ref{contraction-laha}) again, we also get
\begin{eqnarray}
\norm{r^{M,N}_{ts}}_{\cb_{\ga'}} &\leq & c \lln t-s \rrn^\ga \cn[X^{x^M,i} \der P_N(f_i(y^{M,N}));\cac_2^{\ga+\ga'}([0,(l+1)T_0];\cb)] \nonumber\\
&\leq & c \norm{x}_\ga \lln t-s \rrn^\ga \cn[y^{M,N};\cacha_1^{0,{\ga'}}([0,(l+1)T_0])] .\label{est-r-you-2}
\end{eqnarray}
Thus,
\bean
\norm{y^{M,N,\sharp}_{ts}}_{\cb_{\ga'}}&\leq & \norm{r^{M,N}_{tt_q}}_{\cb_{\ga'}}+\norm{r^{M,N}_{st_p}}_{\cb_{\ga'}}+\norm{r^{M,N}_{t_qt_{q-1}}}_{\cb_{\ga'}}+c_{\ga'} \sum_{k=p}^{q-2} \lln t-t_{k+1}\rrn^{-{\ga'}}\norm{r^{M,N}_{t_{k+1}t_k}}_{\cb}\\
&\leq &c \norm{x}_\ga \lcl 1+\cn[y^{M,N};\cacha_1^{0,{\ga'}}([0,(l+1)T_0])] \rcl \cdot\\
& & \hspace{2cm}\lcl \lln t-s \rrn^\ga +\frac{1}{M^{\ga+{\ga'}-1}} \lp \frac{1}{M} \sum_{k=p}^{q-2} \lln t-t_{k+1} \rrn^{-{\ga'}} \rp \rcl\\
&\leq&c \norm{x}_\ga \lcl 1+\cn[y^{M,N};\cacha_1^{0,{\ga'}}([0,(l+1)T_0])] \rcl \lcl \lln t-s \rrn^\ga + \frac{\lln t-s \rrn^{1-{\ga'}}}{M^{\ga+{\ga'}-1}}  \rcl\\
&\leq& c \norm{x}_\ga  \lln t-s \rrn^\ga  \lcl 1+\cn[y^{M,N};\cacha_1^{0,{\ga'}}([0,(l+1)T_0])] \rcl.
\eean

\textit{2$^{\text{nd}}$ subcase}: $t_p \leq s <t < t_{p+1}$. Then $(\delha y^{M,N})_{ts}=X^{x^M,i}_{ts}P_N f_i(y^{M,N}_{t_p})$, so that 
$$\norm{(\delha y^{M,N})_{ts}}_{\cb_{\ga'}} \leq c \norm{x}_\ga \lln t-s \rrn^\ga \lcl 1+\cn[y^{M,N};\cacha_1^{0,{\ga'}}([0,(l+1)T_0])] \rcl.$$

\smallskip

\textit{3$^{\text{rd}}$ subcase}: $t_p \leq s < t_{p+1} \leq t < t_{p+2}$ with $\lln t-s \rrn \leq 1/M$. Just notice that $\norm{(\delha y^{M,N})_{ts}}_{\cb_{\ga'}} \leq \norm{(\delha y^{M,N})_{tt_{p+1}}}_{\cb_{\ga'}}+\norm{(\delha y^{M,N})_{t_{p+1}s}}_{\cb_{\ga'}}$, so that we can go back to the second subcase.

\smallskip

\textit{Conclusion of the 1$^{\text{st}}$ case}: $$\cn[y^{M,N};\cacha_1^{\ga'}([lT_0,(l+1)T_0])] \leq c \norm{x}_\ga T_0^{\ga-{\ga'}} \lcl 1+\cn[y^{M,N};\cacha_1^{0,{\ga'}}([0,(l+1)T_0])] \rcl.$$

\smallskip

\textit{2$^{\text{nd}}$ case}: $s < lT_0 \leq t \leq (l+1)T_0$. One has 
$\norm{(\delha y^{M,N})_{ts}}_{\cb_{\ga'}} \leq \norm{(\delha y^{M,N})_{t,lT_0}}_{\cb_{\ga'}}+\norm{(\delha y^{M,N})_{lT_0,s}}_{\cb_{\ga'}}$, and so, owing to the recurrence assumption,
$$\norm{(\delha y^{M,N})_{ts}}_{\cb_{\ga'}} \leq \lln t-s \rrn^{\ga'} \lcl \cn[y^{M,N};\cacha_1^{\ga'}([lT_0,(l+1)T_0])]+R_l \rcl.$$
The association of the two cases gives
$$\cn[y^{M,N};\cacha_1^{\ga'}([0,(l+1)T_0])] \leq c^1 \norm{x}_\ga T_0^{\ga-{\ga'}} \lcl 1+\cn[y^{M,N};\cacha_1^{0,{\ga'}}([0,(l+1)T_0])] \rcl +R_l.$$
Since, for any $t\in [0,(l+1)T_0]$, $\norm{y^{M,N}_t }_{\cb_{\ga'}} \leq \norm{\psi}_{\cb_{\ga'}}+\cn[y^{M,N};\cacha_1^{\ga'}([0,(l+1)T_0])]$, one deduces
$$
\cn[y^{M,N};\cacha_1^{0,{\ga'}}([0,(l+1)T_0])]
 \leq \norm{\psi}_{\cb_{\ga'}}+2R_l+2c^1 \norm{x}_\ga T_0^{\ga-{\ga'}} \lcl 1+\cn[y^{M,N};\cacha_1^{0,{\ga'}}([0,(l+1)T_0])] \rcl.
$$
To complete the proof of (\ref{hypo-recur}) on $[0,(l+1)T_0]$, it suffices to pick $T_0 $ such that $2c^1 \norm{x}_\ga T_0^{\ga-{\ga'}}=1/2$ and to set
$$R_{l+1}=2\norm{\psi}_{\cb_{\ga'}}+4R_l+1.$$
The bound (\ref{bound-unif-young}) is then easily deduced from the estimate $R_l \leq c\,  4^l \lcl 1+\norm{\psi}_{\cb_{\ga'}} \rcl$.

\end{proof}

\subsection{Space discretization}\label{subsec:disc-spa-you}
This is the final step, that will lead us from $\yba^M$ to $y^{M,N}$. As in the previous subsection, we extend $y^{M,N}$ on $[0,1]$ via (\ref{ext-0-1}) and we use the notation $r^{M,N},y^{M,N,\sharp}$ introduced in Lemma \ref{lem-interm-young}.

\begin{lemma}
For every $M,N\in \N$, if $t_p \leq s <t_{p+1} < \ldots < t_q \leq t < t_{q+1}$ with $\lln t-s \rrn \geq 1/M$, then
$$\norm{y^{M,N,\sharp}_{ts}}_{\cb_{\ga'}} \leq \frac{C(\norm{x}_\ga,\norm{\psi}_{\cb_{\ga'}})}{M^{\ga+{\ga'}-1}} \lln t-s \rrn^{\ga'},$$
for some function $C$ bounded on bounded sets.
\end{lemma}

\begin{proof}
Thanks to the uniform control given by Proposition \ref{prop-unif-y-m-n} and using the estimates (\ref{est-r-you-1})-(\ref{est-r-you-2}), we get
\bean
\lefteqn{\norm{y^{M,N,\sharp}_{ts}}_{\cb_{\ga'}}}\\
 &\leq & \norm{r^{M,N}_{tt_q}}_{\cb_{\ga'}}+\norm{r^{M,N}_{st_p}}_{\cb_{\ga'}}+\norm{r^{M,N}_{t_qt_{q-1}}}_{\cb_{\ga'}}+c_{\ga'} \sum_{k=p}^{q-2} \lln t-t_{k+1}\rrn^{-{\ga'}} \norm{r^{M,N}_{t_{k+1}t_k}}_{\cb}\\
&\leq & c_{x,\psi} \lcl \frac{1}{M^\ga}+\frac{1}{M^{\ga+{\ga'}-1}} \lp \frac{1}{M} \sum_{k=p}^{q-1} \lln t-t_{k+1}\rrn^{-{\ga'}} \rp \rcl\\
&\leq & c_{x,\psi} \lcl \frac{\lln t-s \rrn^{\ga'}}{M^{\ga-{\ga'}}}+\frac{\lln t-s \rrn^{1-{\ga'}}}{M^{\ga+{\ga'}-1}} \rcl \ \leq \ c_{x,\psi} \frac{\lln t-s \rrn^{\ga'}}{M^{\ga+{\ga'}-1}},
\eean
where, for the last inequality, we have used the fact that $1/4 < {\ga'} <1/2$.
\end{proof}

\begin{lemma}
For every $M,N\in \N$, if $t_p \leq s <t_{p+1} < \ldots < t_q \leq t < t_{q+1}$ with $\lln t-s \rrn \geq 1/M$, then
$$\norm{\int_s^t S_{t-u} \, dx^{i,M}_u \, (P_N-\id) f_i(y^{M,N}_u)}_{\cb_{\ga'}} \leq \frac{C(\norm{x}_\ga,\norm{\psi}_{\cb_{\ga'}})}{\la_N^{\ga-{\ga'}}} \lln t-s \rrn^{\ga'},$$
for some function $C$ bounded on bounded sets.
\end{lemma}

\begin{proof}
As $P_N$ commutes with the semigroup, one can write
\begin{multline*}
\int_s^t S_{t-u} \, dx^{i,M}_u \, (P_N-\id) f_i(y^{M,N}_u)\\
=X^{x,i,M}_{ts}(P_N-\id) f_i(y^{M,N}_s)+(P_N-\id)\Laha_{ts}(X^{x,i,M}\der f_i(y^{M,N})).
\end{multline*}
Now, one has
\bean
\lefteqn{\norm{X^{x,i,M}_{ts}(P_N-\id)f_i(y^{M,N}_s)}_{\cb_{\ga'}}}\\
& \leq & c_x\lln t-s \rrn^{\ga'} \norm{(P_N-\id)f_i(y^{M,N}_s)}_{\cb_{2\ga'-\ga}} \quad \text{(use Lemma \ref{lem-defi-x-x})}\\
&\leq & c_x \frac{\lln t-s \rrn^{\ga'}}{\la_N^{\ga-\ga'}} \norm{f_i(y^{M,N}_s)}_{\cb_{\ga'}} \quad \text{(by (\ref{projection}))} \\
& \leq & c_{x,\psi} \frac{\lln t-s \rrn^{\ga'}}{\la_N^{\ga-\ga'}},
\eean
where we have used (\ref{contr-f-y-sobo}) and the uniform control given by Proposition \ref{prop-unif-y-m-n} to get the last inequality. Then, as in (\ref{est-r-you-2}), we get
\bean
\lefteqn{\norm{(P_N-\id) \Laha_{ts}(X^{x,i,M} \der f_i(y^{M,N}))}_{\cb_{\ga'}}}\\
& \leq &\frac{1}{\la_N^{\ga-\ga'}} \norm{\Laha_{ts}(X^{x,i,M} \der f_i(y^{M,N}))}_{\cb_\ga}\\
&\leq &c_{x} \frac{\lln t-s \rrn^{\ga'}}{\la_N^{\ga-\ga'}} \cn[y^{M,N};\cacha_1^{0,{\ga'}}([0,(l+1)T_0])]\ \leq  \ c_{x,\psi} \frac{\lln t-s \rrn^{\ga'}}{\la_N^{\ga-\ga'}}.
\eean

\end{proof}

We are now in a position to prove the main result of this subsection, which, together with Corollary \ref{cor-y-yba}, completes the proof of Theorem \ref{main-result-young}. Indeed, by (\ref{hol-reg}), we know that
$$\cn[y-y^{M,N};\cac_1^{\ga'}(\cb)]\leq c \lcl \cn[y-y^{M,N};\cac_1^{0}(\cb_{\ga'})]+\cn[y-y^{M,N};\cacha_1^{{\ga'}}(\cb_{\ga'})]\rcl.$$

\begin{proposition}
There exists a function $C:(\R^+)^2 \to \R^+$ bounded on bounded sets such that for every $M,N\in \N$,
\begin{multline}
\cn[\yba^M-y^{M,N};\cac_1^{0}(\cb_{\ga'})]+\cn[\yba^M-y^{M,N};\cacha_1^{{\ga'}}(\cb_{\ga'})]\\
 \leq C(\norm{x}_\ga,\norm{\psi}_{\cb_{\ga'}})  \lcl \norm{\psi-P_N \psi}_{\cb_{\ga'}}+\frac{1}{M^{\ga+{\ga'}-1}} \rcl.
\end{multline} 
\end{proposition}

\begin{proof}
As in the previous subsection, we use the short notation
$$\cn[\yba^M-y^{M,N};\cacha_1^{0,{\ga'}}(I)]:=\cn[\yba^M-y^{M,N};\cac_1^{0}(I,\cb_{\ga'})]+\cn[\yba^M-y^{M,N};\cacha_1^{{\ga'}}(I,\cb_{\ga'})].$$
\textit{Local result}. Consider first an interval $I_0=[0,T_0]$, with $T_0$ a time to be precised at the end of this first step, and let $s,t \in [0,T_0]$.

\smallskip

\textit{1$^{\text{st}}$ case}: if $t_p \leq s < t <t_{p+1}$, then
$$\delha (\yba^M-y^{M,N})_{ts}=(\delha \yba^M)_{ts}-X^{x,i,M}_{ts} P_N f_i(y^{M,N}_{t_p}),$$
and hence
$$\norm{\delha (\yba^M-y^{M,N})_{ts}}_{\cb_{\ga'}} \leq c_{\psi,x} \lln t-s \rrn^\ga \leq c_{\psi,x} \frac{\lln t-s \rrn^{\ga'}}{M^{\ga-{\ga'}}} \leq c_{\psi,x} \frac{\lln t-s \rrn^{\ga'}}{M^{\ga+\ga'-1}}.$$

\smallskip

\textit{2$^{\text{nd}}$ case}: if $t_p \leq s < t_{p+1} \leq t < t_{p+2}$, we go back to the previous case by noticing that
$$\norm{\delha (\yba^M-y^{M,N})_{ts}}_{\cb_{\ga'}} \leq \norm{\delha (\yba^M-y^{M,N})_{tt_{p+1}}}_{\cb_{\ga'}}+\norm{\delha (\yba^M-y^{M,N})_{t_{p+1}s}}_{\cb_{\ga'}} .$$

\smallskip

\textit{3$^{\text{rd}}$ case}: $t_p \leq s < t_{p+1} < \ldots < t_q \leq t < t_{q+1}$ with $\lln t-s \rrn \geq 1/M$. Then
\bean
\lefteqn{\delha(\yba^M-y^{M,N})_{ts}}\\
&=& \int_s^t S_{t-u} \, dx^{i,M}_u \, \lc f_i(\yba^M_u)-P_N f_i(y^{M,N}_u) \rc +y^{M,N,\sharp}_{ts}\\
&=&\int_s^t S_{t-u} \, dx^{i,M}_u \, \lc f_i(\yba^M_u)-f_i(y^{M,N}_u) \rc\\
& &\hspace{1cm}+\int_s^t S_{t-u} \, dx^{i,M}_u \, (\id-P_N )f_i(y^{M,N}_u) +y^{M,N,\sharp}_{ts}.
\eean 
According to the two previous lemmas, one can assert that
$$\norm{\int_s^t S_{t-u} \, dx^{i,M}_u \, (\id-P_N )f_i(y^{M,N}_u) +y^{M,N,\sharp}_{ts}}_{\cb_{\ga'}} \leq c_{\psi,x} \lln t-s \rrn^{\ga'} C_{M,N},$$
where we have set $C_{M,N}:= M^{1-(\ga+{\ga'})}+\la_N^{\ga'-\ga}$. Besides, it is not hard to see that
$$\norm{\int_s^t S_{t-u} \, dx^{i,M}_u \, \lc f_i(\yba^M_u)-f_i(y^{M,N}_u) \rc}_{\cb_{\ga'}} \leq c^1_{\psi,x} \lln t-s \rrn^{\ga'} T_0^{\ga-{\ga'}} \cn[\yba^M-y^{M,N};\cacha^{0,{\ga'}}([0,T_0])],$$
for some constant $c^1_{\psi,x}$ that we fix for the rest of the proof.

\smallskip

\noindent
By summing up the three cases, we get
$$\cn[\yba^M-y^{M,N};\cacha_1^{\ga'}([0,T_0];\cb_{\ga'})] \leq c^2_{\psi,x} C_{M,N}+c^1_{\psi,x} T_0^{\ga-{\ga'}} \cn[\yba^M-y^{M,N};\cacha_1^{0,{\ga'}}([0,T_0])] .$$
In order to estimate $\cn[\yba^M-y^{M,N};\cac_1^0([0,T_0],\cb_{\ga'})]$, it now suffices to observe that $\yba^M_s-y^{M,N}_s=\delha(\yba^M-y^{M,N})_{s0}+S_{s}(\psi-P_N \psi)$, and so
\begin{multline*}
\cn[\yba^M-y^{M,N};\cacha_1^{0,{\ga'}}([0,T_0])]\\
\leq \norm{\psi-P_N \psi}_{\cb_{\ga'}}+2 \, c^2_{\psi,x}C_{M,N}+2 \, c^1_{\psi,x} T_0^{\ga-{\ga'}}\cn[\yba^M-y^{M,N};\cacha_1^{0,{\ga'}}([0,T_0])].
\end{multline*}
Thus, pick $T_0$ such that $2\,  c^1_{\psi,x} T_0^{\ga-{\ga'}}=1/2$ to obtain
\begin{equation}\label{result-local-young}
\cn[\yba^M-y^{M,N};\cacha_1^{0,{\ga'}}([0,T_0])] \leq 2\norm{\psi-P_N \psi}_{\cb_{\ga'}}+4\, c^2_{\psi,x}C_{M,N}.
\end{equation}

\smallskip

\textit{Extending the result}: By following the same steps as in the local reasoning, we get, for any $\eta >0$,
\begin{multline*}
\cn[\yba^M-y^{M,N};\cacha_1^{\ga'}([T_0,T_0+\eta];\cb_{\ga'})]\\
 \leq c^2_{\psi,x}C_{M,N}+c^1_{\psi,x} \eta^{\ga-{\ga'}} \cn[\yba^M-y^{M,N};\cacha_1^{0,{\ga'}}([0,T_0+\eta],\cb_{\ga'})] ,
\end{multline*}
which, together with (\ref{result-local-young}), leads to
\begin{multline*}
\cn[\yba^M-y^{M,N};\cacha_1^{\ga'}([0,T_0+\eta];\cb_{\ga'})]\\
 \leq 2\, \norm{\psi-P_N \psi}_{\cb_{\ga'}}+5 \, c^2_{\psi,x}C_{M,N}+c^1_{\psi,x} \eta^{\ga-{\ga'}} \cn[\yba^M-y^{M,N};\cacha_1^{0,{\ga'}}([0,T_0+\eta],\cb_{\ga'})] ,
\end{multline*}
and then
\begin{multline*}
\cn[\yba^M-y^{M,N};\cacha_1^{0,{\ga'}}([0,T_0+\eta])]\\
 \leq 5\norm{\psi-P_N \psi}_{\cb_{\ga'}}+10 \, c^2_{\psi,x}C_{M,N}+2\, c^1_{\psi,x} \eta^{\ga-{\ga'}} \cn[\yba^M-y^{M,N};\cacha_1^{0,{\ga'}}([0,T_0+\eta])] .
\end{multline*}
By taking $\eta=T_0$, we deduce
$$\cn[\yba^M-y^{M,N};\cacha_1^{0,{\ga'}}([0,2T_0])] \leq 10 \, \norm{\psi-P_N \psi}_{\cb_{\ga'}}+20 \, c^2_{\psi,x}C_{M,N}.$$
We repeat the procedure until the whole interval $[0,1]$ is covered.
\end{proof}

\section{Rough case}\label{sec-rough-case}

We now turn to the proof of Theorem \ref{theo-rough}. Thus, let us fix $\ga \in (\frac{1}{3},\frac{1}{2}]$, $\ga'\in (1-\ga,2\ga]$, $\psi \in \cb_{\ga'}$, and suppose that both Assumptions \textbf{(X2)$_\ga$} and \textbf{(F1)$_3$} are satisfied. We will follow (almost) the same steps as in the previous section: we first use pre-existing continuity results to reduce the problem to the study of the Wong-Zakai approximation $\yba^M$, and then lean on a uniform bound for $y^{M,N}$ to control the transition from $\yba^M$ to $y^{M,N}$.

\smallskip

\noindent
Before we start the procedure, let us remind the reader with a few considerations taken from \cite{RHE-glo} on how to understand Equation (\ref{equa-mild}) under Assumption \textbf{(X2)$_\ga$}. As in the Young case, the interpretation is based on the expansion of the ordinary equation: observe that if $\xti$ is a piecewise differentiable path, then
\begin{equation}\label{dec-int-rou-reg}
\int_s^t S_{t-u} \, d\xti^i_u \, f_i(y_u) =X^{\xti,i}_{ts} f_i(y_s)+X^{\xti\xti,ij}_{ts}(f_i'(y_s) \cdot f_j(y_s))+J^y_{ts},
\end{equation} 
where the operator-valued paths $X^{\xti,i},X^{\xti\xti,ij}$ have been defined by (\ref{def-reg-x}) and 
$$J^y_{ts}:=\int_s^t S_{t-u} \, d\xti^i_u \, M^i_{us},$$
with (remember that $a_{ts}:=S_{t-s}-\id$)
\begin{multline}\label{decompo-m}
M^i_{us}:=\int_0^1 dr \, \lc f_i'(y_s+r(\der y)_{us})-f_i'(y_s) \rc \cdot (\der y)_{us}\\
+ \lc a_{us} y_s+\int_s^u S_{u-v} \, d\xti^i_v \, \der(f_i(y))_{vs}+\int_s^u a_{uv} \, dx^j_v \, f_j(y_s) \rc \cdot f_i'(y_s) .
\end{multline}
On top of Lemma \ref{lem-defi-x-x}, one can here rely on the following extension result (which also anticipates the next subsections by introducing the additional path $X^{ax}$):
\begin{lemma}[\cite{RHE}, Proposition 6.3]\label{lem-prol-x-x-x}
The sequence of operator-valued paths
$$X^{ax^M,i}_{ts}:=\int_s^t a_{tu} \, dx^{M,i}_u \quad , \quad resp. \quad X^{x^M x^M,ij}_{ts}:=\int_s^t S_{t-u} \, dx^{M,i}_u \, (\der x^{M,j})_{us},$$
converges to an element $X^{ax,i}$ (resp. $X^{xx,ij}$) with respect to the topology of 
$$\cac_2^{\ga+\ka}(\cl(\cb_{\al+\ka},\cb_\al)) \quad  (\al \geq 0,\ka\in [0,1)),$$
$$\text{resp.} \quad \cac_2^{2\ga-\ka}(\cl(\cb_\al,\cb_{\al+\ka})) \quad (\al \in \R,\ka \in [0,2\ga)).$$
Moreover,
$$\cn[X^{xx,ij};\cac_2^{2\ga-\ka}(\cl(\cb_{\al},\cb_{\al+\ka}))] \leq c_{\al,\ka} \norm{\xrgh}_\ga,$$
$$\cn[X^{x^Mx^M,ij}-X^{xx,ij};\cac_2^{2\ga-\ka}(\cl(\cb_{\al},\cb_{\al+\ka}))] \leq c_{\al,\ka} \{ 1+\norm{\xrgh}_\ga \} \norm{\xrgh-\xrgh^M}_\ga,$$
and the same controls hold for $X^{ax,i}$ in $\cac_2^{\ga+\ka}(\cl(\cb_{\al+\ka},\cb_\al))$. Finally, $X^{ax,i}$ and $X^{xx,ij}$ commute with the projection $P_N$ and satisfy the following algebraic relations:
\begin{equation}
(\delha X^{xx,ij})_{tus}=X^{x,i}_{tu} (\der x^j)_{us} \quad , \quad X^{ax,i}_{ts}=X^{x,i}_{ts}-(\der x^i)_{ts},
\end{equation}
where $X^{x,i}$ is the path given by Lemma \ref{lem-defi-x-x}.
\end{lemma}
Now, from a heuristic point of view, if we go back to the $\ga$-Hölder path $x$ in (\ref{dec-int-rou-reg}), the expression (\ref{decompo-m}) allows to identify $J^y$ as a $\cb$-valued path of order $\mu:=\inf(3\ga,\ga+\ga')>1$. This (partially) accounts for the definition:

\begin{definition}\label{defi-sol-rough}
Let $\ka \in (0,1)$ and $\psi \in \cb_\ka$. A path $y:[0,1]\to \cb_\ka$ is said to be a rough solution of (\ref{equa-mild}) in $\cb_\ka$ if there exists two parameters $\mu >1,\ep >0$ such that
\begin{equation}\label{defi-sol-rough-2}
y_0=\psi \quad \text{and} \quad \delha y-X^{x,i}f_i(y)-X^{xx,ij}(f_i'(y) \cdot f_j(y)) \in \cac_2^\mu(\cb) \cap \cac_2^\ep(\cb_\ka).
\end{equation}
\end{definition}

\begin{remark}
As reported in \cite[Remark 2.7]{RHE-glo}, the consideration of the topology $\cac_2^\ep(\cb_\ka)$ in the definition (\ref{defi-sol-rough-2}) may be surprising at first sight, since it has no counterpart in the standard rough setting (see \cite{Davie}). In fact, this topology arises from the fundamental estimate (\ref{contr-a-delha-2}), as we shall see in the course of the reasoning.
\end{remark}

In accordance with the decomposition (\ref{dec-int-rou-reg}), one has in particular:

\begin{proposition}[\cite{RHE-glo}, Propositions 2.8 and 2.9]
If $x$ is a piecewise differentiable path (resp. a standard Brownian motion) and if the initial condition $\psi$ belongs to $\cb_{\eta}$ with $\eta \in (0,1)$ (resp. $\eta \in (\frac{1}{2},1)$), then the classical (resp. Stratonovich) solution of (\ref{equa-mild}) is also a rough solution in $\cb_\eta$.
\end{proposition}

\begin{remark}
Let us go back here to the Young setting, i.e., when $\ga >1/2$.  In order to connect the above interpretation of (\ref{equa-mild}) with the concept of a solution derived from Proposition \ref{defi-int-young}, observe the following equivalence: under the assumptions of Theorem \ref{theo-young-abs}, a path $y\in \cacha_1^{\ga'}(\cb_{\ga'})$ is solution of (\ref{equa-mild}) (in the sense of Proposition \ref{defi-int-young}) if and only if $y_0=\psi$ and there exists $\mu>1,\ep>0$ such that $\delha y-X^{x,i}f_i(y) \in \cac_2^\mu(\cb) \cap \cac_2^\ep(\cb_{\ga'})$. Indeed, if $y$ is the solution given by Theorem \ref{theo-young-abs}, then, owing to the contraction property (\ref{contraction-laha}), $\delha y-X^{x,i}f_i(y)=\Laha(X^{x,i} \der f_i(y)) \in \cac_2^{\ga+\ga'}(\cb) \cap \cac_2^\ga(\cb_{\ga'})$. On the other hand, if $\delha y-X^{x,i}f_i(y) \in \cac_2^\mu(\cb) \cap \cac_2^\ep(\cb_{\ga'})$ and $z$ is defined by $z_0=\psi$, $\delha z=X^{x,i}f_i(y)+\Laha(X^{x,i}\der f_i(y))$, one has $\delha (y-z)\in \cac_2^{\tilde{\mu}}(\cb)$, with $\tilde{\mu}=\inf(\mu,\ga+\ga') >1$. As $y_0=z_0$, this easily entails $y=z$.
\end{remark}

\subsection{Previous results}
With the above definition in mind, the main result of \cite{RHE-glo} can be summed up in the following way:

\begin{theorem}[\cite{RHE-glo}, Theorem 2.12]\label{theo-exi-rough}
Under the assumptions of Theorem \ref{theo-rough}, Equation (\ref{equa-mild}) admits a unique rough solution in $\cb_{\ga'}$ in the sense of Definition \ref{defi-sol-rough}. Moreover, if $y$ (resp. $\yti$) is the rough solution in $\cb_{\ga'}$ of (\ref{equa-mild}) associated with a path $x$ (resp. $\xti$) that satisfies \textbf{(X2)$_\ga$}, with initial condition $\psi$ (resp. $\tilde{\psi}$) in $\cb_{\ga'}$, then
\begin{multline}\label{conti-appli-ito-rough}
\cn[y-\tilde{y};\cac_1^0([0,1];\cb_{\ga'})]+\cn[y-\yti;\cacha_1^\ga([0,1];\cb)]\\
 \leq C \lp \norm{\xrgh}_\ga, \norm{\tilde{\xrgh}}_\ga,\norm{\psi}_{\cb_{\ga'}},\norm{\tilde{\psi}}_{\ga'} \rp \lcl \norm{\psi-\tilde{\psi}}_{\cb_{\ga'}}+\norm{\xrgh-\tilde{\xrgh}}_\ga \rcl,
\end{multline} 
for some function $C:(\R^+)^4 \to \R^+$ bounded on bounded sets.
\end{theorem}

As in the Young case, we denote by $\yba^M$ the Wong-Zakai solution of (\ref{equa-mild}), which corresponds to the classical (or equivalently rough) solution of the equation when $x$ is replaced with $x^{2^M}$. The continuity result (\ref{conti-appli-ito-rough}) allows us to control the transition from $y$ to $\yba^M$:

\begin{corollary}\label{cor:wz-rough}
Under the assumptions of Theorem \ref{theo-rough}, there exists a function $C:(\R^+)^2 \to \R^+$ bounded on bounded sets such that for every $M\in \N$,
$$\cn[y-\yba^M;\cac_1^0([0,1];\cb_{\ga'})]+\cn[y-\yba^M;\cacha_1^\ga([0,1];\cb)] \leq C(\norm{\xrgh}_\ga,\norm{\psi}_{\cb_{\ga'}}) \cdot \norm{\xrgh-\xrgh^{2^M}}_\ga.$$
\end{corollary}

Now we must notice that the time-discretization of the equation has been analyzed in \cite{RHE-glo} as well. In other words, we already know how to control the difference between $\yba^M$ and the path $y^M$ generated by the following intermediate Milstein scheme: $y^{M}_0= \psi$ and
\begin{equation}\label{intermediate-milstein-scheme}
y^{M}_{t_{k+1}} =S_{t_{k+1}-t_k} y^{M}_{t_k}+ X^{x^{2^M},i}_{t_{k+1}t_k} f_i(y^{M}_{t_k})+X^{x^{2^M} x^{2^M},ij}_{t_{k+1}t_k} \lp f_i'(y^{M}_{t_k}) \cdot f_j(y^{M}_{t_k}) \rp,
\end{equation}
where $t_k=t_k^M=\frac{k}{2^M}$. To express this result, let us denote by $(\Pi^M)$ the sequence of dyadic partitions of $[0,1]$, and introduce the two paths
$$K^M_{ts}:=(\delha y^M)_{ts}-X^{x^{2^M},i}_{ts} f_i(y^{M}_{s}) \quad , \quad J^M_{ts}:=K^M_{ts}-X^{x^{2^M} x^{2^M},ij}_{ts} \lp f_i'(y^{M}_{s}) \cdot f_j(y^{M}_{s}) \rp,$$
for every $s<t\in \Pi^M$. Since all of these paths are defined over the points of the partition only (and not on the whole interval $[0,1]$), we consider in the sequel the discrete version of the norms introduced in Subsection \ref{subsec:tools}. Thus, for any $M\in \N$, we set $\llbracket a,b \rrbracket_M:=[a,b] \cap \Pi^M$ and 
$$\cn[h;\cacha_1^\la(\llbracket t_p^n,t_q^n \rrbracket_M,\cb_{\al,p})] :=\sup_{\substack{t_p^n \leq s <t\leq t_q^n\\ s,t\in \Pi^M}} \frac{\norm{(\delha h)_{ts}}_{\cb_{\al,p}}}{\lln t-s\rrn^\la},$$
We define the quantities 
$$\cn[.;\cac_1^0(\llbracket a,b \rrbracket_M;\cb_{\al,p})]\quad , \quad \cn[.;\cac_2^\la(\llbracket a,b \rrbracket_M;\cb_{\al,p})] \quad , \quad \cn[.;\cac_3^\la(\llbracket a,b \rrbracket_M;\cb_{\al,p})],$$
along the same line.

\begin{proposition}[\cite{RHE-glo}]\label{prop:passage-wz-discret}
Under the assumptions of Theorem \ref{theo-exi-rough}, for every 
$$0< \be < \inf\lp \ga+\ga'-1,\ga-\ga'+\frac{1}{2} \rp,$$
there exists a function $C=C_\be:(\R^+)^2 \to \R^+$ bounded on bounded sets such that for every $M\in \N$,
\begin{equation}\label{passage-wz-discret}
\cn[\yba^M-y^{M};\cac_1^0(\llbracket 0,1\rrbracket_M;\cb_{\ga'})]+\cn[\yba^M-y^{M};\cacha_1^\ga(\llbracket 0,1\rrbracket_M;\cb)] \leq \frac{C( \norm{\xrgh}_\ga,\norm{\psi}_{\cb_{\ga'}} )}{(2^M)^\be},
\end{equation}
where $y^M$ is the path generated by the intermediate Milstein scheme (\ref{intermediate-milstein-scheme}). Moreover, there exists another function $C':(\R^+)^2 \to \R^+$ bounded on bounded sets such that the following uniform control holds: For every $M\in \N$,
\begin{equation}\label{uni-cont-y-m}
\cn[y^{M};\cac_1^0(\llbracket 0,1\rrbracket_M;\cb_{\ga'})]+\cn[y^{M};\cacha_1^\ga(\llbracket 0,1\rrbracket_M;\cb)]+\cn[K^{M};\cac_2^{2\ga}(\llbracket 0,1\rrbracket_M;\cb)]
\leq c_{x,\psi}.
\end{equation}
where $c_{x,\psi}:=C'( \norm{\xrgh}_\ga,\norm{\psi}_{\cb_{\ga'}})$.
\end{proposition}

\begin{remark}
To be more specific about the reference for this result, the bound (\ref{passage-wz-discret}) follows from \cite[Estimate (61)]{RHE-glo}, while (\ref{uni-cont-y-m}) is the consequence of \cite[Estimates (53)-(55)]{RHE-glo}.
\end{remark}

It now remains to study the transition from $y^M$ to $y^{M,N}$, which is the purpose of the two following subsections.

\subsection{A uniform control}
The aim here is to exhibit a uniform estimate for $y^{M,N}$, to which we will extensively appeal in the next subsection. As in the time-discretization procedure, the two following paths will play a prominent role in our reasoning: for every $M,N$ and every $s<t\in \Pi^M$, define
$$K^{M,N}_{ts}:=(\delha y^{M,N})_{ts}-X^{x^{2^M},i}_{ts} P_N f_i(y^{M,N}_{s}),$$
$$J^{M,N}_{ts}:=K^{M,N}_{ts}-X^{x^{2^M} x^{2^M},ij}_{ts}P_N \lp f_i'(y^{M,N}_{s}) \cdot P_Nf_j(y^{M,N}_{s}) \rp.$$

\begin{proposition}\label{prop:cont-unif-y-m-n}
There exists two constants $C_1,C_2 >0$ and an integer $k\geq 2$ such that for every $M ,N\in \N$,
\begin{multline}\label{ineq:contr-uni}
\cn[y^{M,N};\cacha_1^\ga(\llbracket 0,1\rrbracket_M;\cb)]+\cn[y^{M,N};\cac_1^0(\llbracket 0,1\rrbracket_M;\cb_{\ga'})]+\cn[K^{M,N};\cac_2^{2\ga}(\llbracket 0,1\rrbracket_M;\cb)]\\
\leq C_1 \{ 1+\norm{\psi}_{\cb_{\ga'}} \} \exp\lp C_2 \norm{\xrgh}_\ga^k \rp. 
\end{multline}
\end{proposition}

Proposition \ref{prop:cont-unif-y-m-n} is actually a spin-off of the following local bounds on $J^{M,N}$:

\begin{proposition}\label{prop:contr-j-m-n}
Fix $\ep,\mu$ such that
$$\ga+\ga' >\mu >1 \quad , \quad \ga-(\ga'-\frac{1}{2}) >\ep >0.$$
Then there exists two integers $M_0=M_0(\norm{\xrgh}_\ga)$, $N_0=N_0(\norm{\xrgh}_\ga,\norm{\psi}_{\cb_{\ga'}})$ and a time $T_0=T_0(\norm{\xrgh }_\ga) >0$, $T_0\in \Pi^{M_0}$, such that for every $M\geq M_0,N\geq N_0$ and every $k\in \N$,
\begin{equation}\label{cont-j-m-n-1}
\cn[J^{M,N};\cac_2^\mu(\llbracket kT_0,(k+1)T_0 \wedge 1 \rrbracket_M;\cb)] \leq 1+\norm{y^{M,N}_{kT_0}}_{\cb_{\ga'}}
\end{equation}
and
\begin{equation}\label{cont-j-m-n-2}
\cn[J^{M,N};\cac_2^\ep(\llbracket kT_0,(k+1)T_0 \wedge 1 \rrbracket_M;\cb_{\ga'})] \leq 1+\norm{y^{M,N}_{kT_0}}_{\cb_{\ga'}}.
\end{equation}
\end{proposition}

The proof of Proposition \ref{prop:contr-j-m-n} resorts to the following technical lemmas:

\begin{lemma}[\cite{RHE-glo}, Lemma 3.2]\label{lem:contr-a}
Let $\ep >0$ and $\mu >1$. There exists a constant $c=c_{\ep,\mu}$ such that for every $M\in \N$ and any path $A:\cs_2 \to \cb_{\ga'}$ satisfying $A_{t_{k+1}^M t_k^M}=0$ for $k\in \{0,\ldots,2^M-1\}$, one has
\begin{equation}\label{contr-a-delha-1}
\norm{A_{ts}}_{\cb} \leq c \lln t-s \rrn^\mu \cn[\delha A;\cac_3^\mu(\llbracket s,t\rrbracket_M;\cb)]
\end{equation}
and
\begin{equation}\label{contr-a-delha-2}
\norm{A_{ts}}_{\cb_{\ga'}} \leq c \lcl \lln t-s \rrn^\ep+\lln t-s \rrn^{\mu-\ga'} \rcl \lcl  \cn[\delha A;\cac_3^\mu(\llbracket s,t\rrbracket_M;\cb)]+\cn[\delha A;\cac_3^\ep(\llbracket s,t\rrbracket_M;\cb_{\ga'})] \rcl.
\end{equation}
for all $s<t \in \Pi^M$.
\end{lemma}

\begin{lemma}\label{lem:contr-delha-j}
There exists a constant $c>0$ such that for every $M\in \N$ and $s<t \in \Pi^M$,
\begin{multline}\label{contr-delha-j-m-n-1}
\cn[\delha J^{M,N};\cac_3^\mu(\llbracket s,t\rrbracket_M;\cb)]\leq c \lcl 1+\norm{\xrgh}_\ga^2\rcl \lln t-s \rrn^{\ga+\ga'-\mu} \\
\lcl \cn[y^{M,N};\cac_1^0(\llbracket s,t\rrbracket_M);\cb_{\ga'}]+\cn[y^{M,N};\cacha_1^\ga(\llbracket s,t\rrbracket_M);\cb]+\cn[K^{M,N};\cac_2^{2\ga}(\llbracket s,t\rrbracket_M);\cb]\rcl\\
\lcl 1+\la_N^{1/2-\ga'} \cn[y^{M,N};\cac_1^0(\llbracket s,t \rrbracket_M;\cb_{\ga'})]^2 \rcl
\end{multline}
and
\begin{multline}\label{contr-delha-j-m-n-2}
\cn[\delha J^{M,N};\cac_3^\ep(\llbracket s,t \rrbracket_M;\cb_{\ga'})]\leq c \lcl 1+\norm{\xrgh}_\ga^2\rcl  \lln t-s \rrn^{\ga-(\ga'-\frac{1}{2})-\ep}\\
\lcl 1+\cn[y^{M,N};\cac_1^0(\llbracket s,t \rrbracket_M;\cb_{\ga'})]\rcl \lcl 1+\la_N^{1/2-\ga'} \cn[y^{M,N};\cac_1^0(\llbracket s,t \rrbracket_M;\cb_{\ga'})]^2 \rcl.
\end{multline}
\end{lemma}

\begin{proof}
See Appendix A.

\end{proof}

\begin{proof}[Proof of Proposition \ref{prop:contr-j-m-n}]
For the sake of clarity, we write here $x$ for $x^{2^M}$.

\smallskip

\noindent
\textit{Step 1: $k=0$}. This is an iteration procedure over the points of the partition. Assume that both inequalities (\ref{cont-j-m-n-1}) and (\ref{cont-j-m-n-2}) hold true on $\llbracket 0,t_q^M\rrbracket_M$ and $t_{q+1}^M \leq T_0$ ($T_0$ will actually be precised in the course of the reasoning). Then, for every $t\in \llbracket 0,t_q^M\rrbracket_M$, one has, thanks to Lemmas \ref{lem-defi-x-x} and \ref{lem-prol-x-x-x},
\bean
\lefteqn{\norm{y^{M,N}_t}_{\cb_{\ga'}}}\\
&\leq & \norm{J^{M,N}_{t0}}_{\cb_{\ga'}}+\norm{S_{t} \psi}_{\cb_{\ga'}}+\norm{X^{x,i}_{t0} P_N f_i(\psi)}_{\cb_{\ga'}}+\norm{X^{xx,ij}_{t0} P_N (f_i'(\psi) \cdot P_Nf_j(\psi))}_{\cb_{\ga'}}\\
&\leq & 1+2\norm{\psi}_{\cb_{\ga'}}+c \norm{\xrgh}_\ga \lcl t^{\ga-(\ga'-\frac{1}{2})} \norm{f_i(\psi)}_{\cb_{1/2}}+t^{2\ga-(\ga'-\eta)} \norm{f_i'(\psi) \cdot P_Nf_j(\psi)}_{\cb_\eta} \rcl,
\eean
where $\eta \in (0,\frac{1}{2})$ is picked such that $2\ga>\ga'-\eta$. Now, due to (\ref{contr-f-y-sobo}), it holds that $\norm{f_i(\psi)}_{\cb_{1/2}} \leq c\{ 1+\norm{\psi}_{\cb_{\ga'}} \}$ and as in the subsequent estimates (\ref{traitement-prod}) and (\ref{serie-fourier}), one has
$$\norm{f_i'(\psi) \cdot P_Nf_j(\psi)}_{\cb_\eta}\leq c \{ 1+\norm{\psi}_{\cb_{\ga'}} \} \lcl 1+\frac{\norm{\psi}_{\cb_{\ga'}}^2}{\la_N^{\ga'-1/2}}\rcl,$$
so
$$\cn[y^{M,N};\cac_1^0(\llbracket 0,t_q^M\rrbracket_M;\cb_{\ga'})] \leq c^1 \lcl 1+\norm{\xrgh}_\ga\rcl \{ 1+\norm{\psi}_{\cb_{\ga'}} \} \lcl 1+\frac{\norm{\psi}_{\cb_{\ga'}}^2}{\la_N^{\ga'-1/2}}\rcl.$$
At this point, let us introduce an integer $N_0^{1}$ such that for every $N\geq N_0^{1}$, $\frac{\norm{\psi}_{\cb_{\ga'}}^2}{\la_{N}^{\ga'-1/2}}\leq 1$, which entails $$\cn[y^{M,N};\cac_1^0(\llbracket 0,t_q^M\rrbracket_M;\cb_{\ga'})] \leq 2c^1\lcl 1+\norm{\xrgh}_\ga\rcl \{ 1+\norm{\psi}_{\cb_{\ga'}} \}.$$
Besides, if $s<t\in \llbracket 0,t_q^M \rrbracket_M$, one has
\bean
\lefteqn{\norm{(\delha y^{M,N})_{ts}}_{\cb}}\\
&\leq & \norm{J^{M,N}_{ts}}_{\cb}+\norm{X^{x,i}_{ts} P_N f_i(y^{M,N}_s)}_{\cb}+\norm{X^{xx,ij}_{ts} P_N (f_i'(y^{M,N}_s) \cdot P_Nf_j(y^{M,N}_s))}_{\cb}\\
&\leq &  \lln t-s \rrn^\ga \lcl 1+\norm{\psi}_{\cb_{\ga'}}+c\norm{\xrgh}_\ga \rcl,
\eean
and
\bean
\norm{K^{M,N}_{ts}}_{\cb} &\leq & \norm{J^{M,N}_{ts}}_{\cb}+\norm{X^{xx,ij}_{ts} P_N (f_i'(y^{M,N}_s) \cdot P_Nf_j(y^{M,N}_s))}_{\cb}\\
&\leq &  \lln t-s \rrn^{2\ga} \lcl 1+\norm{\psi}_{\cb_{\ga'}}+c\norm{\xrgh}_\ga \rcl.
\eean
By using (\ref{contr-delha-j-m-n-1}), we deduce that for every $N\geq N_0^{1}$ (remember that $t_{q+1}^M \leq T_0$),
\bean
\lefteqn{\cn[\delha J^{M,N};\cac_3^\mu(\llbracket 0,t_{q+1}^M \rrbracket_M;\cb)]}\\
&\leq & c\, T_0^{\ga+\ga'-\mu}\lcl 1+\norm{\xrgh}_\ga^3\rcl \{ 1+\norm{\psi}_{\cb_{\ga'}} \}  \lcl 1+\frac{\cn[y^{M,N};\cac_1^0(\llbracket 0,t_q^M \rrbracket_M;\cb_{\ga'})]^2}{\la_N^{\ga'-1/2}} \rcl\\
&\leq & c\,  T_0^{\ga+\ga'-\mu}\lcl 1+\norm{\xrgh}_\ga^5 \rcl \{ 1+\norm{\psi}_{\cb_{\ga'}} \}  \lcl 1+\frac{\norm{\psi}_{\cb_{\ga'}}^2}{\la_N^{\ga'-1/2}} \rcl\\
&\leq & c\,  T_0^{\ga+\ga'-\mu}\lcl 1+\norm{\xrgh}_\ga^5 \rcl \{ 1+\norm{\psi}_{\cb_{\ga'}} \}  ,
\eean
and in the same way, according to (\ref{contr-delha-j-m-n-2}),
$$\cn[\delha J^{M,N};\cac_3^\ep(\llbracket 0,t_{q+1}^M \rrbracket_M;\cb_{\ga'})]\leq c\, T_0^{\ga-(\ga'-\frac{1}{2})-\ep}\lcl 1+\norm{\xrgh}_\ga^5\rcl\lcl 1+\norm{\psi}_{\cb_{\ga'}} \rcl  .$$
Then, thanks to (\ref{contr-a-delha-1}) and (\ref{contr-a-delha-2}) (applied to $A=J^{M,N}$), we obtain, for every $N\geq N_0^{1}$,
$$\cn[J^{M,N};\cac_2^\mu(\llbracket 0,t^M_{q+1} \rrbracket_M;\cb)] \leq c^2 \lcl 1+\norm{\xrgh}_\ga^5\rcl \lcl T_0^{\ga+\ga'-\mu}+T_0^{\ga-(\ga'-\frac{1}{2})-\ep}\rcl \lcl 1+\norm{\psi}_{\cb_{\ga'}} \rcl,$$
and
$$\cn[J^{M,N};\cac_2^\ep(\llbracket 0,t^M_{q+1} \rrbracket_M;\cb_{\ga'})] \leq c^2 \lcl 1+\norm{\xrgh}_\ga^5\rcl  \lcl T_0^{\ga+\ga'-\mu}+T_0^{\ga-(\ga'-\frac{1}{2})-\ep}\rcl \lcl 1+\norm{\psi}_{\cb_{\ga'}} \rcl.$$
Consider now a real $T_0^\ast >0$ such that 
$$c^2\lcl 1+\norm{\xrgh}_\ga^5\rcl  \lcl (T_0^\ast)^{\ga+\ga'-\mu}+(T_0^\ast)^{\ga-(\ga'-\frac{1}{2})-\ep}\rcl\leq 1$$
and let $M_0=M_0(\norm{\xrgh}_\ga)$ be an integer such that $1/(2^{M_0})\leq T_0^\ast$. We fix $T_0$ in the non empty set $(0,T_0^\ast) \cap \Pi^{M_0}$ so as to retrieve the expected controls, namely: for every $M\geq M_0,N\geq N_0^{1}$,
$$\cn[J^{M,N};\cac_2^\mu (\llbracket 0,t_{q+1}^M\rrbracket_M;\cb)] \leq 1+\norm{\psi}_{\ga'},\quad \cn[J^{M,N};\cac_2^\ep (\llbracket 0,t_{q+1}^M\rrbracket_M;\cb_{\ga'})] \leq 1+\norm{\psi}_{\ga'},$$
which completes Step 1, that is to say the proof of (\ref{cont-j-m-n-1}) and (\ref{cont-j-m-n-2}) on $\llbracket 0,T_0 \rrbracket_M$.

\smallskip

\noindent
\textit{Step 2: $k=1$}. We henceforth fix $M\geq M_0$. With the same arguments as in Step 1, we first deduce, if both controls (\ref{cont-j-m-n-1}) and (\ref{cont-j-m-n-2}) are checked on $\llbracket T_0,t_q^M\rrbracket_M$ (with $t_{q+1}^M \leq 2T_0$),
\begin{equation}\label{contr-step-2}
\cn[y^{M,N};\cac_1^0(\llbracket T_0,t_q^M\rrbracket_M;\cb_{\ga'})] \leq c^1 \lcl 1+\norm{\xrgh}_\ga\rcl  \lcl 1+\norm{y^{M,N}_{T_0}}_{\cb_{\ga'}} \rcl \lcl 1+\frac{\norm{y^{M,N}_{T_0}}_{\cb_{\ga'}}^2}{\la_N^{\ga'-1/2}}\rcl.
\end{equation}
Remember that for every $N\geq N_0^{1}$, $\norm{y^{M,N}_{T_0}}_{\cb_{\ga'}} \leq 2c^1 \{ 1+\norm{\xrgh}_\ga\} \{ 1+\norm{\psi}_{\cb_{\ga'}} \}$. Consequently, we introduce an integer $N_0^{2} \geq N_0^{1}$ such that for every $N\geq N_0^{2}$, 
$$\frac{\norm{y^{M,N}_{T_0}}_{\cb_{\ga'}}^2}{\la_N^{\ga'-1/2}}\leq \frac{(2c^1 \lcl 1+\norm{\xrgh}_\ga\rcl \{ 1+\norm{\psi}_{\cb_{\ga'}} \})^2}{\la_N^{\ga'-1/2}}\leq 1,$$
and (\ref{contr-step-2}) entails, for every $N\geq N_0^{2}$,
$$\cn[y^{M,N};\cac_1^0(\llbracket T_0,t_q^M\rrbracket_M;\cb_{\ga'})] \leq 2c^1 \lcl 1+\norm{\xrgh}_\ga\rcl \lcl 1+\norm{y^{M,N}_{T_0}}_{\cb_{\ga'}} \rcl.$$
Then, with the same estimates as in Step 1 (replace $\llbracket 0,t_q^M\rrbracket_M$ with $\llbracket T_0,t_q^M\rrbracket_M$ and $\psi$ with $y^{M,N}_{T_0}$), we get, for every $N\geq N_0^{2}$,
$$\cn[J^{M,N};\cac_2^\mu(\llbracket T_0,t^M_{q+1} \rrbracket_M;\cb)] \leq   1+\norm{y^{M,N}_{T_0}}_{\cb_{\ga'}},$$
and the same bound holds for $\cn[J^{M,N};\cac_2^\ep(\llbracket T_0,t^M_{q+1} \rrbracket_M;\cb_{\ga'})]$, which completes the proof of (\ref{cont-j-m-n-1}) and (\ref{cont-j-m-n-2}) on $\llbracket T_0,2T_0 \rrbracket_M$.

\smallskip

\noindent
We repeat the procedure until Step $L$, where $L=L(\norm{\xrgh}_\ga)$ is the smallest integer such that $L T_0 \geq 1$.

\end{proof}

Once endowed with the estimates of Proposition \ref{prop:contr-j-m-n}, the proof of Proposition \ref{prop:cont-unif-y-m-n} is derived from a standard patching argument. Note in particular that $T_0$ can be taken such as $T_0=c \norm{\xrgh}_\ga^{-k}$ for some constant $c>0$ and some integer $k\geq 2$, which accounts for the bound (\ref{ineq:contr-uni}). We refer the reader to the proof of \cite[Theorem 2.10]{RHE-glo} for further details on the procedure.

\subsection{Space discretization}\label{subsec:disc-spa-rou}
We are now in a position to compare $y^M$ with $y^{M,N}$. To this end, let us introduce the following intermediate quantity: for every $s<t\in \Pi^M$,
\begin{multline*}
\cn[y^M-y^{M,N};\cq(\llbracket s,t \rrbracket_M)]:= \cn[y^M-y^{M,N};\cac_1^0(\llbracket s,t\rrbracket_M;\cb_{\ga'})]\\
+\cn[y^M-y^{M,N};\cacha_1^\ga(\llbracket s,t\rrbracket_M;\cb)]+\cn[K^M-K^{M,N};\cac_2^{2\ga}(\llbracket s,t\rrbracket_M;\cb)].
\end{multline*}

\begin{lemma}\label{lem:space}
For every $\eta \in (0,\ga+\ga'-1)$, there exists a function $C_\eta : (\R^+)^2 \to \R^+$ bounded on bounded sets such that for every $M,N\in \N$ and every $s<t \in \Pi^M$,
$$\cn[\delha(J^M-J^{M,N});\cac_3^{\ga+\ga'-\eta}(\llbracket s,t \rrbracket_M;\cb)] \leq c_{x,\psi}  \lcl \lln t-s \rrn^\eta \cn[y^M-y^{M,N};\cq(\llbracket s,t \rrbracket_M)]+\frac{1}{\la_N^{\eta}} \rcl,$$
and
$$\cn[\delha(J^M-J^{M,N});\cac_3^{\ga-\eta}(\llbracket s,t \rrbracket_M;\cb_{\ga'})] \leq c_{x,\psi}  \lcl \lln t-s \rrn^\eta \cn[y^M-y^{M,N};\cq(\llbracket s,t \rrbracket_M)]+\frac{1}{\la_N^{\eta}} \rcl,$$
where $c_{x,\psi}:=C_\eta(\norm{\xrgh}_\ga,\norm{\psi}_{\cb_{\ga'}})$.
\end{lemma}

\begin{proof}
See Appendix A.
\end{proof}

\begin{proof}[Proof of Theorem \ref{theo-rough}]
For the sake of clarity, we write here $x$ for $x^{2^M}$. Consider a time $T_1 \in \llbracket 0,1\rrbracket_M$. For any $t\in \llbracket 0,T_1 \rrbracket_M$, one has
\begin{multline*}
y^M_t-y^{M,N}_t=\lc \psi-P_N \psi \rc+\lc X^{x,i}_{t0}f_i(\psi)-X^{x,i}_{to}P_N f_i(P_N \psi) \rc\\
+\lc X^{xx,ij}_{t0}\lp f_i'(\psi) \cdot f_j(\psi)\rp-X^{xx,ij}_{t0}P_N \lp f_i'(P_N \psi) \cdot P_N f_j(P_N \psi) \rp \rc+\lc J^M_{t0}-J^{M,N}_{t0} \rc.
\end{multline*}
Thanks to Lemma \ref{lem:contr-a} (applied to $A=J^M-J^{M,N}$) and Lemma \ref{lem:space}, we already know that
$$\norm{J^M_{t0}-J^{M,N}_{t0}}_{\cb_{\ga'}} \leq  c_{x,\psi} \lcl T_1^\ga \cn[y^M-y^{M,N};\cq(\llbracket 0,T_1 \rrbracket_M)]+\la_N^{-\eta} \rcl.$$
Then
\bean
\lefteqn{\norm{X^{x,i}_{t0}f_i(\psi)-X^{x,i}_{t0}P_N f_i(P_N \psi)}_{\cb_{\ga'}}}\\
& \leq & \norm{X^{x,i}_{t0} \lc f_i(\psi)-f_i(P_N \psi) \rc}_{\cb_{\ga'}}+\norm{X^{x,i}_{t0}(\id-P_N)f_i(P_N \psi)}_{\cb_{\ga'}}\\
&\leq & c_{x,\psi}\lcl \norm{\psi-P_N \psi}_{\cb_{\ga'}}+\la_N^{-\eta} \rcl
\eean
and with similar arguments, we get
\begin{multline*}
\norm{X^{xx,ij}_{t0}\lp f_i'(\psi) \cdot f_j(\psi)\rp-X^{xx,ij}_{t0}P_N \lp f_i'(P_N \psi) \cdot P_N f_j(P_N \psi) \rp}_{\cb_{\ga'}}\\
\leq c_{x,\psi} \lcl \norm{\psi-P_N \psi}_{\cb_{\ga'}}+\la_N^{-\eta} \rcl,
\end{multline*}
so that
\begin{multline*}
\cn[y^M-y^{M,N};\cac_1^0(\llbracket 0,T_1 \rrbracket_M;\cb_{\ga'})] \\
\leq c_{x,\psi} \lcl T_1^\ga \cn[y^M-y^{M,N};\cq(\llbracket 0,T_1\rrbracket_M)]+\norm{\psi-P_N \psi}_{\cb_{\ga'}} +\la_N^{-\eta}\rcl.
\end{multline*}
Let us now analyze (in $\cb$) the decomposition: for every $s<t \in \llbracket 0,T_1 \rrbracket_M$,
\begin{multline*}
\delha(y^M-y^{M,N})_{ts}=\lc X^{x,i}_{ts}f_i(y^M_s)-X^{x,i}_{ts}P_N f_i(y^{M,N}_s) \rc\\
+\lc X^{xx,ij}_{ts}\lp f_i'(y^M_s) \cdot f_j(y^M_s)\rp-X^{xx,ij}_{ts}P_N \lp f_i'(y^{M,N}_s) \cdot P_N f_j(y^{M,N}_s) \rp \rc+\lc J^M_{ts}-J^{M,N}_{ts} \rc.
\end{multline*}
According to Lemmas \ref{lem:contr-a} and \ref{lem:space}, 
\begin{equation}\label{j-m-j-m-n}
\norm{J^M_{ts}-J^{M,N}_{ts}}_{\cb} \leq c_{x,\psi} \lln t-s \rrn^{2\ga} \lcl T_1^{\ga'-\ga} \cn[y^M-y^{M,N};\cq(\llbracket 0,T_1 \rrbracket_M)]+\la_N^{-\eta} \rcl.
\end{equation}
Moreover,
\bean
\lefteqn{\norm{X^{x,i}_{ts}f_i(y^M_s)-X^{x,i}_{ts}P_N f_i(y^{M,N}_s)}_\be}\\
&\leq & c_x \lln t-s \rrn^\ga  \norm{y^M_s-y^{M,N}_s}_\cb+\norm{X^{x,i}_{ts}(\id-P_N)f_i(y^{M,N}_s)}_\cb\\
&\leq & c_{x,\psi} \lln t-s \rrn^\ga \lcl T_1^\ga \cn[y^M-y^{M,N};\cq(\llbracket 0,T_1 \rrbracket_M)]+\norm{\psi-P_N\psi}_{\cb_{\ga'}}+\la_N^{-\eta} \rcl
\eean
and this kind of argument leads to
\begin{multline*}
\cn[y^M-y^{M,N};\cacha_1^\ga(\llbracket 0,T_1 \rrbracket_M;\cb_{\ga'})] \\
\leq c_{x,\psi} \lcl T_1^\ga \cn[y^M-y^{M,N};\cq(\llbracket 0,T_1\rrbracket_M)]+\norm{\psi-P_N \psi}_{\cb_{\ga'}} +\la_N^{-\eta}\rcl.
\end{multline*}
Finally,
\begin{multline*}
K^M_{ts}-K^{M,N}_{ts}=\\
\lc X^{xx,ij}_{ts}\lp f_i'(y^M_s) \cdot f_j(y^M_s)\rp-X^{xx,ij}_{ts} P_N \lp f_i'(y^{M,N}_s) \cdot P_Nf_j(y^{M,N}_s) \rp \rc+\lc J^M_{ts}-J^{M,N}_{ts} \rc,
\end{multline*}
and thanks to (\ref{j-m-j-m-n}), this decomposition easily allows us to conclude that
\begin{multline*}
\cn[y^M-y^{M,N};\cq(\llbracket 0,T_1 \rrbracket_M)] \\
\leq c^1_{x,\psi} \lcl T_1^{\ga'-\ga} \cn[y^M-y^{M,N};\cq(\llbracket 0,T_1\rrbracket_M)]+\norm{\psi-P_N \psi}_{\cb_{\ga'}} +\la_N^{-\eta}\rcl.
\end{multline*}
Let $T_1^\ast >0$ and $M_0\in \N$ such that $c^1_{x,\psi} (T_1^\ast)^{\ga-\ga'}=\frac{1}{2}$ and $(0,T_1^\ast) \cap \Pi^{M_0} \neq \emptyset$. The time $T_1$ is now fixed in $(0,T_1^\ast) \cap \Pi^{M_0}$ so as to retrieve, for every $M\geq M_0$,
$$\cn[y^M-y^{M,N};\cq(\llbracket 0,T_1 \rrbracket_M)] 
\leq 2c^1_{x,\psi} \lcl \norm{\psi-P_N \psi}_{\cb_{\ga'}} +\la_N^{-\eta}\rcl.
$$
It is readily checked that the same reasoning (with the same constants) holds on any interval $\llbracket kT_1,(k+1) T_1 \wedge 1\rrbracket_M$ and it entails that
$$\cn[y^M-y^{M,N};\cq(\llbracket kT_1,(k+1)T_1\wedge 1 \rrbracket_M)] 
\leq 2c^1_{x,\psi} \lcl \norm{y^M_{kT_1}-y^{M,N}_{kT_1}}_{\cb_{\ga'}} +\la_N^{-\eta}\rcl.
$$
As $T_1$ only depends on $x$ and $\psi$, it follows from a standard patching argument that
$$\cn[y^M-y^{M,N};\cac_1^0(\llbracket 0,1\rrbracket_M;\cb_{\ga'})]+\cn[y^M-y^{M,N};\cacha_1^\ga(\llbracket 0,1\rrbracket_M;\cb)] \leq c_{x,\psi}\lcl \norm{\psi-P_N\psi}_{\cb_{\ga'}}+ \la_N^{-\eta}\rcl,$$
which, together with the results of Corollary \ref{cor:wz-rough} and Proposition \ref{prop:passage-wz-discret}, completes the proof of (\ref{main-cont-rough}), since
\begin{multline*}
\cn[y-y^{M,N};\cac_1^\ga(\llbracket 0,1 \rrbracket_M;\cb)]\\
\leq c \lcl \cn[y-y^{M,N};\cacha_1^\ga(\llbracket 0,1 \rrbracket_M;\cb)]+\cn[y-y^{M,N};\cac_1^0(\llbracket 0,1 \rrbracket_M;\cb_{\ga'})] \rcl.
\end{multline*}

\end{proof}

\section{Appendix A}

Let us go back here to the technical proofs that have been left in abeyance in Section \ref{sec-rough-case}.

\begin{proof}[Proof of Lemma \ref{lem:contr-delha-j}]

For the sake of clarity, we write here $x$ for $x^{2^M}$ and $y$ for $y^{M,N}$. One has
\begin{multline}\label{decompo-j-m-n-1}
(\delha J^{M,N})_{tus}=X^{x,i}_{tu} P_N \der(f_i(y))_{us}-X^{x,i}_{tu} (\der x^j)_{us} P_N \lp f_i'(y_s) \cdot P_Nf_j(y_s) \rp\\
+X^{xx,ij}_{tu} P_N \der(f_i'(y) \cdot P_Nf_j(y))_{us},
\end{multline}
which easily entails
\begin{equation}\label{decompo-j-m-n-2}
(\delha J^{M,N})_{tus}=I^{M,N}_{tus}+II^{M,N}_{tus}+III^{M,N}_{tus}+IV^{M,N}_{tus},
\end{equation}
with
\begin{align*}
& I^{M,N}_{tus}:=X^{x,i}_{tu} P_N \lp \int_0^1 dr \, f_i'(y_s+r(\der y)_{us}) \cdot K^{M,N}_{us} \rp,\\
& II^{M,N}_{tus}:=X^{x,i}_{tu} P_N \lp \int_0^1 dr \, f_i'(y_s+r(\der y)_{us}) \cdot \lc a_{us} y_s+X^{ax,j}_{us} P_N f_j(y_s) \rc  \rp,\\
& III^{M,N}_{tus}:=X^{x,i}_{tu} P_N \lp \int_0^1 dr \, \lc f_i'(y_s+r(\der y)_{us})-f_i'(y_s) \rc (\der x^j)_{us}  \cdot P_N f_j(y_s) \rp,\\
& IV^{M,N}_{tus}:=X^{xx,ij}_{tu} P_N \der \lp f_i'(y) \cdot P_N f_j(y) \rp_{us}.
\end{align*}
First, $\norm{I^{M,N}_{tus}}_\cb \leq c \norm{x}_\ga \lln t-u \rrn^\ga \norm{K^{M,N}_{us}}_\cb$ and
\bean
\norm{II^{M,N}_{tus}}_{\cb} &\leq & c \lcl 1+\norm{\xrgh}_\ga^2\rcl \lln t-u \rrn^\ga \lcl \lln u-s \rrn^{\ga'} \norm{y_s}_{\cb_{\ga'}}+\lln u-s \rrn^{\ga+1/2} \norm{f_i(y_s)}_{\cb_{1/2}} \rcl\\
&\leq & c \lcl 1+\norm{\xrgh}_\ga^2\rcl \lln t-s \rrn^{\ga+\ga'} \lcl 1+\norm{y_s}_{\cb_{\ga'}} \rcl.
\eean
Then
\bean
\lefteqn{\norm{III^{M,N}_{tus}}_\cb}\\
 &\leq & c \norm{x}_\ga^2 \lln t-s \rrn^{2\ga} \norm{(\der y)_{us}}_\cb \norm{P_Nf_i(y_s)}_{L^\infty}\\
&\leq & c \norm{x}_\ga^2 \lln t-s \rrn^{2\ga} \lcl  \norm{(\delha y)_{us}}_\cb+\lln u-s \rrn^{\ga'} \norm{y_s}_{\cb_{\ga'}}  \rcl \lcl 1+\norm{(P_N-\id) f_i(y_s)}_{L^\infty}\rcl
\eean
and
\begin{eqnarray}
\norm{(P_N-\id) f_i(y_s)}_{L^\infty} \  \leq \  c \norm{(P_N-\id)f_i(y_s)}_{\cb_{1/2}} &\leq & \frac{c}{\la_N^{\ga'-1/2}} \norm{f_i(y_s)}_{\cb_{\ga'}} \nonumber\\
& \leq & \frac{c}{\la_N^{\ga'-1/2}} \lcl 1+\norm{y_s}_{\cb_{\ga'}}^2 \rcl.\label{serie-fourier}
\end{eqnarray}
Finally,
\bean
\norm{IV^{M,N}_{tus}}_\cb &\leq & c \norm{\xrgh}_\ga \lln t-u \rrn^{2\ga} \norm{(\der y)_{us}}_\cb \lcl 1+\norm{P_Nf_j(y_s)}_{L^\infty} \rcl\\
&\leq & c \norm{\xrgh}_\ga \lln t-u \rrn^{2\ga} \norm{(\der y)_{us}}_\cb \lcl 1+ \frac{\norm{y_s}_{\cb_{\ga'}}^2}{\la_N^{\ga'-1/2}} \rcl.
\eean
Going back to (\ref{decompo-j-m-n-2}), these estimates yield (\ref{contr-delha-j-m-n-1}). To get (\ref{contr-delha-j-m-n-2}), we resort to the decomposition (\ref{decompo-j-m-n-1}) and we observe that (for instance)
\bean
\norm{X^{x,i}_{tu} P_N f_i(y_u)}_{\cb_{\ga'}} &\leq & c \norm{x}_\ga \lln t-u \rrn^{\ga-(\ga'-\frac{1}{2})} \norm{f_i(y_u)}_{\cb_{1/2}} \\
&\leq & c \norm{x}_\ga \lln t-s \rrn^{\ga-(\ga'-\frac{1}{2})} \lcl 1+\norm{y_u}_{\cb_{\ga'}} \rcl,
\eean
and for every $\eta \in (\ga'-\ga,\frac{1}{2})$,
\begin{eqnarray}
\lefteqn{\norm{X^{x,i}_{tu} (\der x^j)_{us} P_N(f_i'(y_s) \cdot P_N f_j(y_s))}_{\cb_{\ga'}}} \nonumber\\
&\leq & c \norm{x}_\ga^2 \lln t-u\rrn^{\ga-(\ga'-\eta)} \lln u-s \rrn^\ga \norm{f_i'(y_s) \cdot P_Nf_j(y_s)}_{\cb_\eta} \nonumber\\
&\leq & c \norm{x}_\ga^2 \lln t-s \rrn^{2\ga-(\ga'-\eta)} \lcl 1+\norm{y_s}_{\cb_{\ga'}} \rcl \lcl 1+\norm{P_N f_j(y_s)}_{L^\infty} \rcl,\label{traitement-prod}
\end{eqnarray}
where, to get the last estimate, we have used the property (\ref{pp-1}). Together with (\ref{serie-fourier}), this leads to (\ref{contr-delha-j-m-n-2}).

\end{proof}

\begin{proof}[Proof of Lemma \ref{lem:space}]
Observe first that $\delha J^M$ can be decomposed as in (\ref{decompo-j-m-n-1}) or as in (\ref{decompo-j-m-n-2}) by suppressing in both expressions the projection operator $P_N$. In order to estimate $\norm{\delha (J^M-J^{M,N})_{tus}}_\cb$, we rely on the decomposition (\ref{decompo-j-m-n-2}) and its equivalent for $J^M$, with $I^M$ instead of $I^{M,N}$, etc. Write for instance
\bean
\lefteqn{I^M_{tus}-I^{M,N}_{tus}}\\
&=& X^{x,i}_{tu} \lp \int_0^1 dr \, \lc f_i'(y^M_s+r(\der y^M)_{us})-f_i'(y^{M,N}_s+r(\der y^{M,N})_{us}) \rc \cdot K^M_{us} \rp\\
& & +X^{x,i}_{tu} \lp \int_0^1 dr \, f_i'(y^{M,N}_s+r(\der y^{M,N})_{us}) \cdot \lc K^M_{us}-K^{M,N}_{us} \rc \rp\\
& & +X^{x,i}_{tu}(\id -P_N) \lp \int_0^1 dr \, f_i'(y^{M,N}_s+r(\der y^{M,N})_{us}) \cdot K^{M,N}_{us} \rp \ =: \ I^{(1)}_{tus}+I^{(2)}_{tus}+I^{(3)}_{tus}.
\eean
Owing to the uniform estimate (\ref{uni-cont-y-m}) and the continuous inclusion $\cb_{\ga'} \subset L^{\infty}$, one has first
\bean
\norm{I^{(1)}_{tus}}_{\cb}& \leq & c_{x,\psi} \lln t-s \rrn^{3\ga} \lcl \norm{y^M_s-y^{M,N}_s}_{L^\infty}+\norm{y^M_u-y^{M,N}_u}_{L^\infty} \rcl\\
&\leq & c_{x,\psi} \lln t-s \rrn^{3\ga} \cn[y^M-y^{M,N};\cq(\llbracket s,t\rrbracket_M)].
\eean
Then clearly $\norm{I^{(2)}_{tus}}_{\cb} \leq c_x \lln t-s \rrn^{3\ga} \cn[K^M-K^{M,N};\cac_2^{2\ga}(\llbracket s,t \rrbracket_M;\cb)]$ and
\bean
\norm{I^{(3)}_{tus}}_\cb &\leq & c_x \lln t-u \rrn^{\ga-\eta}\norm{(\id-P_N)\lp \int_0^1 dr \, f_i'(y^{M,N}_s+r(\der y^{M,N})_{us}) \cdot K^{M,N}_{us}\rp}_{\cb_{-\eta}}\\
&\leq & c_x \lln t-u \rrn^{\ga-\eta}\la_N^{-\eta} \norm{K^{M,N}_{us}}_\cb \ \leq \ c_{x,\psi}\lln t-s \rrn^{3\ga-\eta} \la_N^{-\eta},
\eean
where, for the last estimate, we have used the uniform control (\ref{ineq:contr-uni}). The other terms $II,III,IV$ of (\ref{decompo-j-m-n-2}) can be handled with similar arguments. Let us only elaborate on the estimate of $\norm{X^{xx,ij}_{tu} P_N \lp f_i'(y^{M,N}_u) \cdot (\id-P_N) \der f_j(y^{M,N})_{us}\rp}_\cb$, which may be a little bit more tricky than the others. Indeed, one must here appeal to the property (\ref{pp-2}) to get
\bean
\lefteqn{\norm{X^{xx,ij}_{tu}P_N \lp f_i'(y^{M,N}_u) \cdot (\id-P_N) \der f_j(y^{M,N})_{us}\rp}_\cb}\\
&\leq & c_x \lln t-u \rrn^{2\ga-\eta} \norm{f_i'(y^{M,N}_u) \cdot (\id-P_N) \der f_j(y^{M,N})_{us}}_{\cb_{-\eta}}\\
&\leq & c_x \lln t-u\rrn^{2\ga-\eta} \norm{f_i'(y^{M,N}_u)}_{\cb_{\ga'}} \norm{(\id-P_N) \der f_j(y^{M,N})_{us}}_{\cb_{-\eta}}\\
&\leq & c_{x,\psi} \lln t-u \rrn^{2\ga-\eta} \la_N^{-\eta} \norm{\der f_j(y^{M,N})_{us}}_\cb \ \leq \ c_{x,\psi} \lln t-s \rrn^{3\ga-\eta} \la_N^{-\eta}.
\eean
As far as $\norm{\delha(J^M-J^{M,N})_{tus}}_{\cb_{\ga'}}$ is concerned, one can start from the decomposition (\ref{decompo-j-m-n-1}) and observe for instance that
\bean
\lefteqn{\norm{X^{x,i}_{tu} f_i(y^M_u)-X^{x,i}_{tu} P_N f_i(y^{M,N}_u)}_{\cb_{\ga'}}}\\
&\leq & \norm{X^{x,i}_{tu} \lc f_i(y^M_u)-f_i(y^{M,N}_u)\rc }_{\cb_{\ga'}}+\norm{X^{x,i}_{tu} (\id-P_N)f_i(y^{M,N}_u)}_{\cb_{\ga'}}\\
&\leq & c_x \lln t-u \rrn^\ga \norm{\int_0^1 dr \, f_i'(y^{M,N}_u+r(y^M_u-y^{M,N}_u)) \cdot (y^M_u-y^{M,N}_u)}_{\cb_{\ga'}}\\
& & +c_x \lln t-u\rrn^{\ga-\eta} \norm{(\id-P_N) f_i(y^{M,N}_u)}_{\cb_{\ga'-\eta}}\\
& \leq & c_{x,\psi} \lcl \lln t-s \rrn^\ga \norm{y^M_u-y^{M,N}_u}_{\cb_{\ga'}}+\lln t-s \rrn^{\ga-\eta} \la_N^{-\eta} \rcl.
\eean
The other terms steming from (\ref{decompo-j-m-n-1}) can be estimated along the same lines.
\end{proof}

\section{Appendix B: Implementation}

We would like to conclude by insisting on the simplicity of the two algorithms (\ref{euler-scheme}) and (\ref{milstein-scheme}) as far as implementation is concerned. To this end, we focus on the case where $x=B$ is a fBm with Hurst index $H\in (1/3,1)$ and $x^M$ is its linear interpolation. As reported in Subsection \ref{subsec:gen-assump}, we know that $x^M$ satisfies Assumption \textbf{(X2)$_\ga$} (and accordingly Assumption \textbf{(X1)$_\ga$}) for any $\ga \in (\frac{1}{3},H)$. We also restrict the implementation to the case where $A$ stands for the Laplacian operator, i.e., $A=\partial_{\xi\xi}$. This allows us to consider the orthonormal basis of eigenfunctions $e_n:=\sqrt{2} \sin (\pi n \xi)$ ($n\in \N^\ast$) associated with the eigenvalues $\la_n:= \pi n^2$.

\subsection{Young case }\label{sub-sec-appli-young}
Consider first the Euler scheme (\ref{euler-scheme}) as $H>1/2$. By setting $Y^{M,N,l}_{t_k}=\left\langle Y^{M,N}_{t_k},e_l \right\rangle$, the formula reduces to an elementary iteration procedure: for $k\in \{0,\ldots,M\}$, $l\in \{1,\ldots, N\}$,
\begin{equation}
Y^{M,N,l}_{t_{k+1}}=e^{-\la_l/M} Y^{M,N,l}_{t_k}+\frac{M}{\la_l} \lcl 1-e^{-\la_l/M}\rcl \sum_{i=1}^m (\der B^i)_{t_{k+1}t_k} \left\langle f_i(Y^{M,N}_{t_k}), e_l \right\rangle.
\end{equation}

The following Matlab code is a possible implementation of this iterative procedure, for which we have taken $m=1$ and 
\begin{equation}\label{defi-f-k}
\psi(\xi)=\sqrt{2} \sin(\pi \xi) \ \ (\xi \in [0,1]), \quad f(x)=\frac{10\cdot(1-x)}{1+x^2} \ \ (x \in \R).
\end{equation}

\smallskip

\noindent 
To be more specific, the procedure simulates the evolution in time of the functional-valued path $Y^{M,N}$. At each step, the Fourier coefficients $\left\langle f_i(Y^{M,N}_{t_k}), e_l \right\rangle$ are computed by means of the discrete sinus transform function \textbf{dst} (and its inverse \textbf{idst}), according to the approximation formula
$$\lla f_i(Y^{M,N}_{t_k}),e_l \rra=\int_0^1 d\xi \, f_i(Y^{M,N}_{t_k}(\xi)) e_l(\xi) \approx \frac{1}{N} \sum_{n=0}^N f_i\lp Y^{M,N}_{t_k}\lp \frac{n}{N} \rp \rp e_l\lp \frac{n}{N}\rp.$$
As for the fBm increments, they are computed via (an appropriately rescaled version of) the Matlab-function \textbf{wfbm}, which leans on the decomposition of the process in a wavelet basis, following the method proposed by Abry and Sellan in \cite{ab-sel}. Let us also notice that the action of the semigroup is likely to be qualified by turning the heat semigroup $S^{A}$ into $S_t=S^{A}_{\ka t}$, for some parameter $\ka$ (we have picked $\ka=100$ in Figures 1 and 2 below). The theoretical study contained in Section \ref{sec-young-case} remains obviously valid for the modified system.

\lstset{language=Scilab}
\begin{lstlisting}

function[l]=eigval(N)
l=[]; for i=1:N, l(i)=(pi*i)^2;end

function[S]=semigr(M,N,l,kappa)
S=[];for i=1:N, v(i)=exp(-l(i)^2/(kappa*M));end

function=simulyoung(H,M,N,k,kappa)
l=eigval(N);S=semigr(M,N,l,kappa);
B=(1/M)^H*wfbm(H,M+1);
A=[1,zeros(1,N-1)];
u=zeros(1,N);fy=zeros(1,N);
for i=1:M
    u=dst(A(i,:));fy=0.5*idst(k*(1-u)./(1+u.^2));
    A(i+1,:)=S.*A(i,:)
             +((kappa./l).*(1-S))*M*(B(i+1)-B(i)).*fy;
    end
E=[];for j=1:M+1, E(j,:)=dst(A(j,:));end
plot(linspace(0,1,N+2),
     [0,dst([1,zeros(1,N-1)]),0]);
F(1)=getframe;for p=1:M
    plot(linspace(0,1,N+2),[0,E(p+1,:),0]);
    hold off;
    F(p+1)=getframe;end
movie(F,1,2)
    
\end{lstlisting}

\subsection{Rough case}\label{sub-sec-appli-rough}

Pick $H\in (1/3,1/2]$ and consider the Milstein scheme (\ref{milstein-scheme}). By projecting $Y^{M,N}$ onto $e_l$, one retrieves this time
\begin{multline}\label{algo-cas-rou}
Y^{M,N,l}_{t_{k+1}}=e^{-\la_l/2^M} Y^{M,N,l}_{t_k}+\frac{2^M}{\la_l} \lcl 1-e^{-\la_l/2^M}\rcl \sum_{i=1}^m (\der B^i)_{t_{k+1}t_k} \left\langle f_i(Y^{M,N}_{t_k}), e_l \right\rangle\\
+(2^M)^2 \sum_{i,j=1}^m (\der B^i)_{t_{k+1}t_k}(\der B^j)_{t_{k+1}t_k} \lp \int_{t_k}^{t_{k+1}} e^{-\la_l(t_{k+1}-u)} du \, (u-t_k)\rp\\
 \lla  P_N f_j(Y^{M,N}_{t_k}) \cdot f_i'(Y^{M,N}_{t_k}),e_l \rra. 
\end{multline}
The computation of the Fourier coefficients $\lla f_i(Y^{M,N}_{t_k}),e_l \rra$ can be implemented with the discrete sinus transform, as in the Young case. As for the computation of 
$$\lla P_N f_j(Y^{M,N}_{t_k}) \cdot f_i'(Y^{M,N}_{t_k}),e_l \rra,$$
it can be achieved with the same idea, starting from the approximation
\begin{multline*}
\lla P_Nf_j(Y^{M,N}_{t_k}) \cdot f_i'(Y^{M,N}_{t_k}),e_l \rra \\
\approx  \frac{1}{N^2} \sum_{n=0}^N \sum_{p=0}^N \sum_{m=0}^N e_l\lp \frac{n}{N}  \rp e_p\lp \frac{n}{N} \rp e_p\lp \frac{m}{N}\rp   f_i'\lp Y^{M,N}_{t_k}\lp \frac{n}{N} \rp \rp f_j\lp Y^{M,N}_{t_k}\lp \frac{m}{N}\rp \rp.
\end{multline*}
These considerations easily lead to the construction of an algorithm for (\ref{algo-cas-rou}).

\begin{figure}[!ht]\label{figure-un}
\epsfig{figure=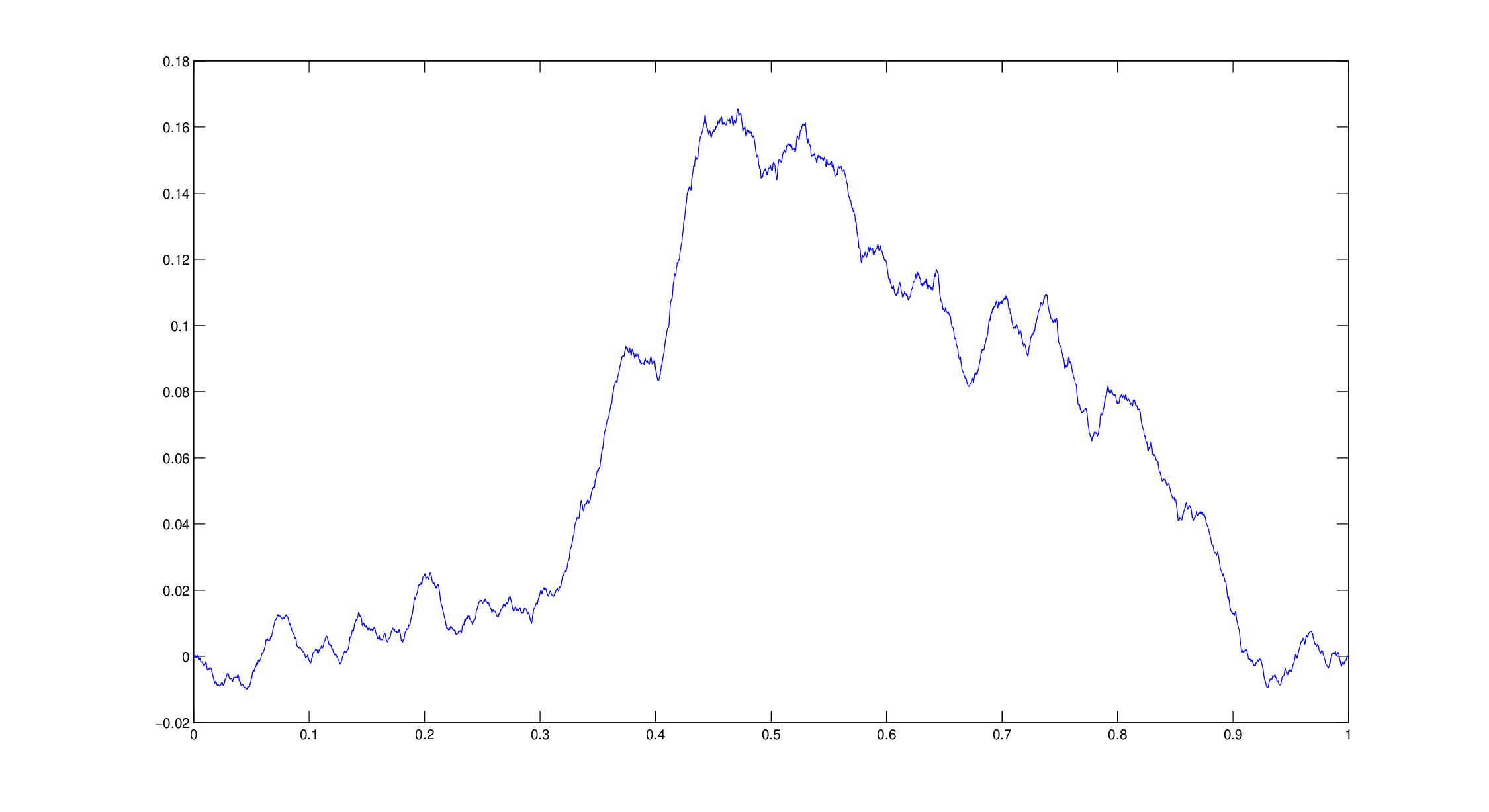,width=15cm, height=9.2cm}
\caption{A simulation of $t \mapsto Y^{M,N}_t(\frac{1}{2})$ (for the conditions described by (\ref{defi-f-k})) via the Euler scheme when $H=0.8$. Here, $M=N=3000$.} 
\end{figure}

\begin{figure}[!ht]\label{figure-deux}
\epsfig{figure=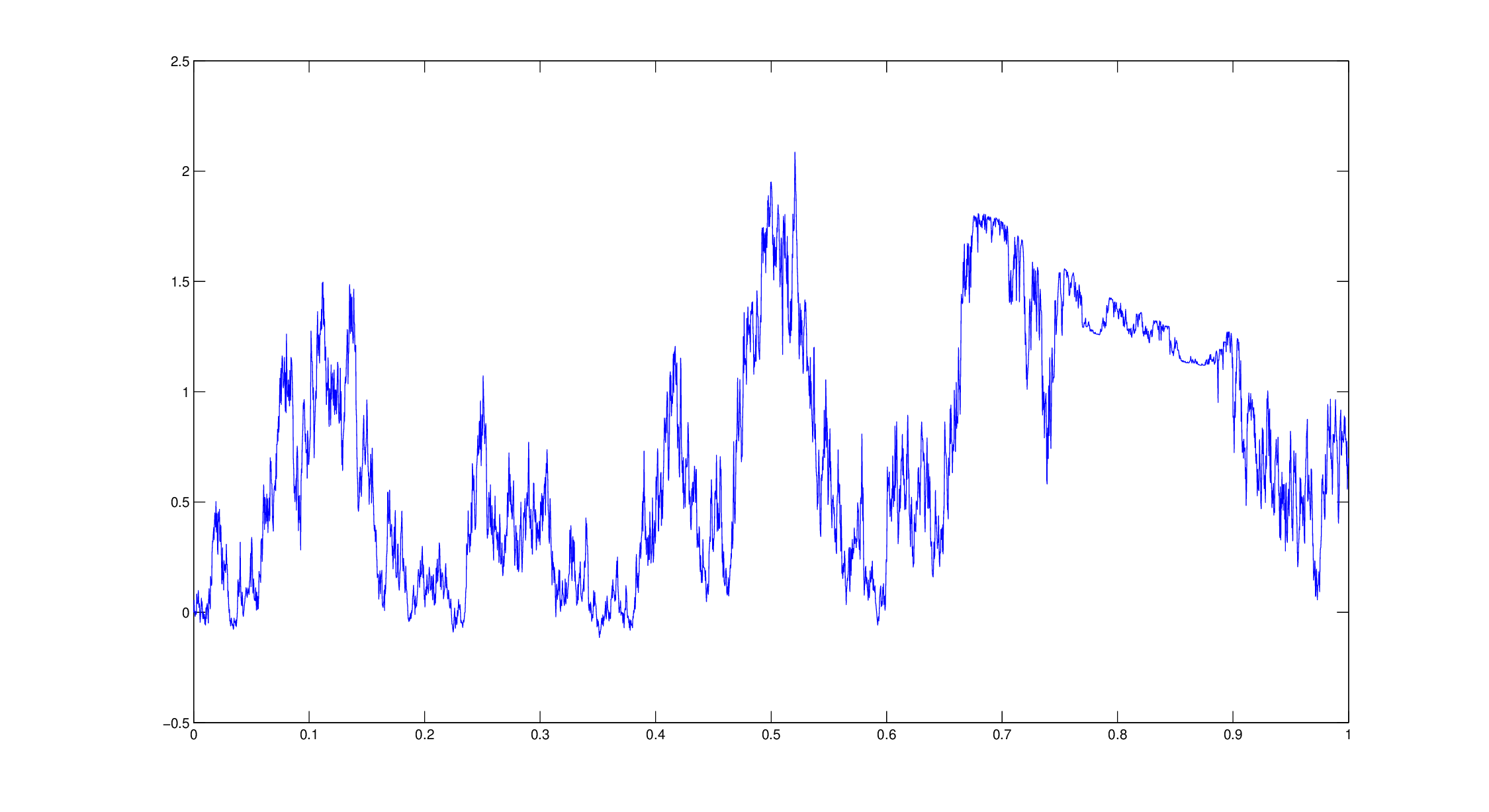,width=15cm, height=9.2cm}
\caption{A simulation of $t \mapsto Y^{M,N}_t(\frac{1}{2})$ (for the conditions described by (\ref{defi-f-k})) via the Milstein scheme when $H=0.35$. Here, $M=N=3000$.} 
\end{figure}

\

\newpage

\textbf{Acknowledgement}

\

I am very grateful to two anonymous referees for their careful reading and their suggestions, which greatly helped me to improve the presentation of these results.

\bibliography{mabiblio-scheme-rhe-updated}{}
\bibliographystyle{plain}

\end{document}